\documentclass[reqno, a4paper, 12pt]{amsart}

\usepackage{comment}

\usepackage[british]{babel}

\usepackage{bbm}
\usepackage[top=2.50cm,left=2.65cm,right=2.65cm,bottom=2.50cm]{geometry}
\usepackage{amsfonts}
\usepackage{amsmath,amsthm,amscd,mathrsfs,amssymb,graphicx,enumerate,
	dsfont}
\usepackage{xypic}
\usepackage{hyperref}
\usepackage[capitalise]{cleveref}
\usepackage[usenames,dvipsnames]{xcolor}
\hypersetup{colorlinks=true,citecolor=NavyBlue,linkcolor=Maroon,urlcolor=Orange}

\allowdisplaybreaks

\newtheorem{theorem}{Theorem}[section]
\newtheorem{lemma}[theorem]{Lemma}
\newtheorem{corollary}[theorem]{Corollary}
\newtheorem{conjecture}[theorem]{Conjecture}
\newtheorem{proposition}[theorem]{Proposition}

\theoremstyle{definition}
\newtheorem{remark}[theorem]{Remark}
\newtheorem{example}[theorem]{Example}
\newtheorem{definition}[theorem]{Definition}

\crefname{theorem}{Theorem}{Theorems}
\Crefname{theorem}{Theorem}{Theorems}
\crefname{conjecture}{Conjecture}{Conjectures}
\Crefname{conjecture}{Conjecture}{Conjectures}
\crefname{lemma}{Lemma}{Lemmas}
\Crefname{lemma}{Lemma}{Lemmas}
\crefname{corollary}{Corollary}{Corollaries}
\Crefname{corollary}{Corollary}{Corollaries}
\crefname{proposition}{Proposition}{Propositions}
\Crefname{proposition}{Proposition}{Propositions}
\crefname{definition}{Definition}{Definitions}
\Crefname{definition}{Definition}{Definitions}
\crefname{remark}{Remark}{Remarks}
\Crefname{remark}{Remark}{Remarks}
\crefname{example}{Example}{Examples}
\Crefname{example}{Example}{Examples}
\Crefname{equation}{}{}
\crefname{equation}{}{}

\let\oldconjecture\conjecture
\renewcommand{\conjecture}{%
  \crefalias{theorem}{conjecture}%
  \oldconjecture
}

\let\olddefinition\definition
\renewcommand{\definition}{%
  \crefalias{theorem}{definition}%
  \olddefinition
}

\let\oldlemma\lemma
\renewcommand{\lemma}{%
  \crefalias{theorem}{lemma}%
  \oldlemma
}

\let\oldcorollary\corollary
\renewcommand{\corollary}{%
  \crefalias{theorem}{corollary}%
  \oldcorollary
}

\let\oldproposition\proposition
\renewcommand{\proposition}{%
  \crefalias{theorem}{proposition}%
  \oldproposition
}

\let\oldremark\remark
\renewcommand{\remark}{%
  \crefalias{theorem}{remark}%
  \oldremark
}

\let\oldexample\example
\renewcommand{\example}{%
  \crefalias{theorem}{example}%
  \oldexample
}

\usepackage{amsmath,tikz-cd}
\usepackage{setspace}
\usepackage{amsfonts, mathrsfs}
\usepackage{amssymb}
\usepackage{mathtools}
\usepackage{enumitem}

\usepackage[utf8]{inputenc}

\usepackage{csquotes}
\usepackage[style=alphabetic, backend=biber, maxalphanames=5, maxbibnames=99, minbibnames=99, url=false]{biblatex}
\AtEveryBibitem{%
  \clearfield{issn}
}

\newcommand{\w}{\wedge}

\newcommand{\IN}{\ensuremath{\mathbb{N}}}
\newcommand{\IZ}{\ensuremath{\mathbb{Z}}}
\newcommand{\IQ}{\ensuremath{\mathbb{Q}}}
\newcommand{\IR}{\ensuremath{\mathbb{R}}}
\newcommand{\IC}{\ensuremath{\mathbb{C}}}
\newcommand{\IP}{\ensuremath{\mathbb{P}}}

\newcommand{\id}{\mathrm{id}}

\DeclareMathOperator{\supp}{supp}

\newcommand{\Ox}{\ensuremath{\mathcal{O}}}

\DeclareMathOperator{\im}{im}

\renewcommand{\comment}[1]{}

\DeclareMathOperator{\proj}{Proj}

\DeclareMathOperator{\Gr}{Gr}
\DeclareMathOperator{\chow}{CH}
\DeclareMathOperator{\gdch}{GDCH}

\DeclareMathOperator{\sym}{Sym}
\DeclareMathOperator{\pr}{pr}

\begin{document}

\title{On the Chow ring of double EPW quartics}
\author{Carl Mazzanti}

\begin{abstract}
    Double EPW quartics are hyperkähler varieties of dimension 4, first introduced by Iliev, Kapustka, Kapustka, and Ranestad. The general double EPW quartic is isomorphic to a moduli space of twisted sheaves on a $K3$ surface. 
    They have a rich geometry: they are equipped with an anti-symplectic involution and are related to conics in Verra fourfolds in the same way Fano varieties of lines on cubic fourfolds are related to cubic fourfolds themselves.
    In this work, we exploit this geometry to establish general conjectures about algebraic cycles on hyperkähler varieties in the case of double EPW quartics.
\end{abstract}

\maketitle

\tableofcontents

\section{Introduction}

\subsection{Algebraic cycles on hyperkähler varieties}

Algebraic cycles on hyperkähler varieties have been the subject of much research.
Following the work of Beauville and Voisin on the Chow ring of a $K3$ surface in \cite{BV}, Beauville conjectured that the conjectural Bloch--Beilinson filtration admits a canonical splitting in the case of hyperkähler varieties \cite{beauville07}. 
As a concrete consequence, Beauville conjectured that the subring generated by divisors maps injectively to the cohomology via the cycle class map for any hyperkähler variety \cite{beauville07}. 
Voisin considered a generalisation of this conjecture in \cite{voisin08}.

\begin{conjecture}[Beauville--Voisin, see \Cref{conj:BV}]
    Let $T$ be a hyperkähler variety. Then the cycle class map restricted to the subring generated by divisors and Chern classes,
    $$\mathrm{cl}\colon\left\langle\chow^1(T), c_i(T)\right\rangle\to H^\ast(T,\IQ),$$
    is injective.
\end{conjecture}

Another question studied following \cite{BV} is the Franchetta conjecture for hyperkähler varieties. 
O'Grady noticed in \cite{og2} that the Beauville--Voisin class constructed in \cite{BV} is \emph{generically defined}, see \Cref{def:gen-defined}, and asked whether any generically defined zero-cycle on a $K3$ surface is a multiple of the Beauville--Voisin class. 
This was generalised to higher-dimensional hyperkähler varieties in \cite{flv19}, where the Franchetta property was introduced and conjectured to hold for locally complete families of hyperkähler varieties.

\begin{conjecture}[Franchetta conjecture, see \Cref{conj:franchetta}]
    Let $\mathcal{F}$ be the moduli stack of a locally complete family of hyperkähler varieties and $\mathcal{T}\to\mathcal{F}$ its universal family. 
    Then it satisfies the Franchetta property, meaning that for any $b\in\mathcal{F}$ the cycle class map restricted to the subring of generically defined cycles,
    $$\mathrm{cl}\colon\gdch^\ast(\mathcal{T}_b)\to H^\ast(\mathcal{T}_b,\IQ),$$
    is injective.
\end{conjecture}

The interaction of the Franchetta conjecture with the Beauville--Voisin conjecture has been studied in \cite{bvf}.

These conjectures are open in general, but there are two cases where significant progress has been made: the Fano variety of lines on a cubic fourfold and Hilbert schemes of points on $K3$ surfaces. 
The Beauville--Voisin conjecture has been proven in the case of the Fano variety of lines and for $S^{[n]}$ for small $n$ in \cite{voisin08}.

The Franchetta property is known to hold for $\mathcal{F}_1(\mathcal{X})\to\mathcal{M}$, the universal family of Fano varieties of lines on cubic fourfolds, and for $\mathcal{S}^{n/\mathcal{F}_{2d}}\to\mathcal{F}_{2d}$, the universal family of powers of $K3$ surfaces for small $n$ and low degree $d$, by \cite{flv19}, even though the latter is not a locally complete family.
See \Cref{rmk:non-complete} for details on the Franchetta property in the case of non-locally complete families.
In addition to this, a multiplicative splitting of the Chow ring of these varieties was constructed in \cite{shenvial}, verifying some results predicted by the Bloch--Beilinson conjecture. 

The Chow ring of moduli spaces of stable objects on possibly twisted $K3$ surfaces, of which double EPW quartics are examples, is less well understood. 
The more geometric techniques used to study the Beauville--Voisin conjecture for Hilbert schemes of points on $K3$ surfaces and Fano varieties of lines are not available in this setting, complicating the study of their Chow rings.

\subsection{The Beauville--Voisin conjecture for double EPW quartics}

The aim of this paper is to complete the picture in the case of double EPW quartics.
These are hyperkähler fourfolds of $K3^{[2]}$-type admitting an anti-symplectic involution and were first constructed by Iliev, Kapustka, Kapustka, and Ranestad in \cite{ikkr17}.
They form a divisor in the moduli space of $K3^{[2]}$-type varieties.
While they are constructed via Lagrangian degeneracy loci, the general double EPW quartic admits a description in terms of $(1,1)$-conics on Verra fourfolds and as a moduli space of stable objects on a twisted $K3$ surface.
This rich geometry makes it possible to study their cycle-theoretic properties in depth.

We prove the following generalisation of the Beauville--Voisin conjecture proposed in \cite{bvf}. For $K3^{[2]}$-type hyperkähler varieties, this is equivalent to both the Beauville--Voisin and the Franchetta conjecture, see \Cref{subsec:BVF} for more details.

The Beauville--Voisin conjecture is known in a handful of cases \cite{BV,voisin08, fu-BV, laterveervial, laterveer23, bvf}, see \Cref{subsec:BVF} for more details.
The Franchetta conjecture, on the other hand, is known for fewer hyperkähler varieties and some Fano varieties \cite{pavic-shen-yin,fu-laterveer,lu2025ogradysgeneralizedfranchettaconjecture,flv19,flv21}, see \Cref{sec:franchetta} for more details.

\begin{theorem}[\Cref{thm:BV}]\label{intro:bv}
    Let $Y$ be a double EPW quartic. Then the Beauville--Voisin--Franchetta conjecture holds for $Y$, meaning that the cycle class map restricted to the subring generated by divisors and generically defined cycles,
    $$\mathrm{cl}\colon\left\langle \chow^1(Y),\gdch^\ast(Y)\right\rangle \to H^\ast(Y,\IQ),$$
    is injective.
\end{theorem}

In particular, this establishes the existence of a special zero-cycle $o_Y\in\chow_0(Y)$, called the Beauville--Voisin class, satisfying $\langle\chow^1(Y),\gdch^\ast(Y)\rangle_0 =\IQ o_Y$.
We prove that this is not only a zero-cycle, but is represented by a point on a constant cycle subvariety.

\begin{theorem}[\Cref{thm:constant-cycle}]\label{intro:constant-cycle}
    Let $Y$ be a general double EPW quartic and $Z$ the fixed locus of its involution $\iota$. Then $Z$ is a constant cycle surface in $Y$ and the Beauville--Voisin class $o_Y$ is represented by any point lying on $Z$.
\end{theorem}

\subsection{Parallels to Fano varieties of lines on cubic fourfolds}

As in the case of double EPW sextics, or more generally hyperkähler fourfolds with an anti-symplectic involution, we expect the Bloch--Beilinson filtration on zero-cycles to split in the following way, where the superscript indicates (anti-)invariance with respect to the involution~$\iota$ on $Y$,
$$\chow_0(Y)=\IQ o_Y\oplus\chow_0(Y)^-\oplus \chow_0(Y)^+_{\mathrm{hom}}.$$
This bears resemblance to the case of Fano varieties of lines on cubic fourfolds, where it is known that the group of zero-cycles admits an eigenspace decomposition with respect to Voisin's self-rational map \cite{shenvial}.

Furthermore, the general double EPW quartic $Y$ is associated to a unique Verra fourfold $X$, first studied in \cite{verra}, which is a Fano variety of Picard rank $2$ and $K3$ type.
This is similar to both the case of the Fano variety of lines on a cubic fourfold and to the case of double EPW sextics, which are also associated to Fano varieties of $K3$ type: they are associated to cubic fourfolds and Gushel--Mukai fourfolds, respectively.
More precisely, the Hilbert scheme $F(X)$ of conics in $X$ admits the structure of a $\IP^1$-bundle over $Y$. 
The universal conic induces a correspondence between $Y$ and $X$, similar to the universal line inducing a correspondence between the Fano variety of lines on a cubic fourfold and the cubic fourfold itself. 
Both in the case of Fano varieties of lines and of double EPW sextics, it is known that the middle graded part of the group of zero-cycles is isomorphic to the group of homologically trivial $1$-cycles on the associated fourfold \cite{shenvial, zhang24}.
We prove the analogous statement and compare the Chow ring of $Y$ with that of $X$.

\begin{theorem}[\Cref{thm:chow-iso,prop:h^2,cor:chow^3_strong}]\label{intro:chow-iso}
    Let $Y$ be a general double EPW quartic, $X$ the Verra fourfold associated to $Y$, and $P$ the universal conic.
    We have isomorphisms 
    $$\chow_0(Y)^-\xrightarrow[\cong]{P_\ast} \chow_1(X)_{\mathrm{hom}}\xrightarrow[\cong]{P^t_\ast} \chow^2(Y)^-_{\mathrm{hom}}$$
    induced by the universal conic.
    The inverse of the composition is given (up to a non-zero scalar) by intersecting with $h^2$.
    Furthermore, 
    $$\chow^3(Y)^-_{\mathrm{hom}}=h\cdot\chow^2(Y)^-_{\mathrm{hom}}\oplus g\cdot \chow^2(Y)^-_{\mathrm{hom}},$$
    where $g$ is the other invariant divisor, see \Cref{subsec:motives} for the definition.
\end{theorem}

We also show that Voisin's filtration, introduced in \cite{voisin16}, agrees with the expected splitting of the conjectural Bloch--Beilinson filtration.

\begin{proposition}[\Cref{prop:filtration}]
    Let $Y$ be a double EPW quartic. 
    Then Voisin's filtration agrees with the filtration obtained from the involution $\iota$,
    $$
    \IQ o_Y\subset \IQ o_Y\oplus\chow_0(Y)^-\subset \chow_0(Y).
    $$
\end{proposition}

Due to the surjectivity of $H^0(Y,\Omega_Y^2)\otimes H^0(Y,\Omega_Y^2)\to H^0(Y,\Omega_Y^4)$, the Bloch--Beilinson conjecture predicts that the multiplication map $F^2\chow^2(Y)\otimes F^2\chow^2(Y)\to F^4\chow_0(Y)$ is surjective on the deepest part of the conjectural filtration. 
While the existence of the Bloch--Beilinson filtration is not known in general, this was shown for what it is expected to be in the case of Fano varieties of lines on cubic fourfolds  \cite[Thm.~20.2]{shenvial}.
We would expect that $F^2\chow^2(Y)=\chow^2(Y)^-_{\mathrm{hom}}$ in the case of double EPW quartics and double EPW sextics.
We prove it in this setting.

\begin{theorem}[\Cref{them:multiplicativity+0}, \Cref{thm:A-multiplicativity}]\label{intro:multiplicativity}
    Let $Y$ be a general double EPW quartic or a general double EPW sextic. Then the intersection pairing
    $$\chow^2(Y)^-_{\mathrm{hom}}\otimes \chow^2(Y)^-_{\mathrm{hom}}\to \chow_0(Y)^+_{\mathrm{hom}}$$
    is surjective.
\end{theorem}

\subsection{Structure}

The paper is structured as follows. In \Cref{sec:construction} we review the three constructions of double EPW quartics, see \Cref{thm:construction}:
\begin{enumerate}[label=\alph*.]
    \item\label{a} as the base of a $\IP^1$-bundle given by the Hilbert scheme of $(1,1)$-conics on a Verra fourfold,
    \item\label{b} double covers of Lagrangian degeneracy loci, and
    \item\label{c} as moduli spaces of twisted sheaves on a $K3$ surface.
\end{enumerate}
It is meant to be referred back to whenever a different construction is used. 
In \Cref{sec:franchetta} we discuss generically defined cycles and the Franchetta property and prove that it holds for double EPW quartics. 
\Cref{sec:incidence} contains the technical details on the incidence variety relying on construction \ref{a}, which we prove satisfies relations in the Chow ring that we will use in the subsequent sections. 
In \Cref{sec:constant-cycle-surface} we show that the fixed locus is a constant cycle subvariety by computing certain intersection products.
In \Cref{sec:multiplicativity} we use similar methods to prove that the multiplication is surjective on the deepest part of the Bloch--Beilinson filtration.
Finally, we use results from \Cref{sec:incidence,sec:constant-cycle-surface} and construction \ref{b} to prove the Beauville--Voisin--Franchetta conjecture for double EPW quartics in \Cref{sec:BV}.
The Appendix~\ref{appendix} contains adaptations of the results from \Cref{sec:constant-cycle-surface,sec:multiplicativity} to double EPW sextics, for which these were not known yet.

\subsection{Conventions}
We work over the complex numbers.
A variety refers to a separated, integral scheme of finite type over $\IC$ and a subvariety to a closed, reduced, possibly reducible, subscheme of a variety, unless stated otherwise. 
A hyperkähler variety refers to a smooth, projective variety that is simply connected in the analytic topology and carries a holomorphic symplectic form that is unique up to scalars, i.e. $H^0(T,\Omega_T^2)=\IC\sigma$.
Points in varieties refer to closed points, unless otherwise specified.
Chow rings will be assumed to have rational coefficients and $\chow^\ast(T)_{\mathrm{hom}}$ always refers to homologically trivial cycles with respect to singular cohomology.
A general point in a variety refers to a point contained in a non-empty open subset, while a very general point in a variety refers to a point contained in the non-empty complement of a countable union of subvarieties. 
We refer to \cite{motives} for details on Chow motives.

A glossary of notation can be found at the end of this article.

\subsection{Acknowledgements}

The author thanks Robert Laterveer for suggesting this project and Charles Vial for invaluable conversations over the course of completing this work.
The author is funded by the Deutsche Forschungsgemeinschaft (DFG, German Research Foundation) -- Project-ID 491392403 -- TRR 358. 

\section{The construction of double EPW quartics}\label{sec:construction}

Introduced by Iliev, Kapustka, Kapustka, and Ranestad in \cite{ikkr17}, double EPW quartics $Y$ are hyperkähler fourfolds of $K3^{[2]}$-type that form a $19$-dimensional family. 
They carry a polarisation $h$ of degree $4$ with respect to the Beauville--Bogomolov form~$q$, i.e. $q(h,h)=4$, and an anti-symplectic involution $\iota$ with invariant lattice isometric to~$U(2)$.
Furthermore, they admit two Lagrangian fibrations.

They were first constructed in \cite{ikkr17}, where three different constructions were given: one via Lagrangian degeneracy loci, one via conics in Verra fourfolds and one via moduli spaces of twisted sheaves on K3 surfaces. 
We review these constructions in this section. 
While the details of the third construction will not be used in this article, we include it for future reference.
The second and third constructions can be skipped on first read and referred back to when needed.
We first summarise the results.
\begin{theorem}[\cite{ikkr17}]\label{thm:construction}
    There is a $19$-dimensional family $\mathcal{U}$ of hyperkähler varieties~$Y$ of $K3^{[2]}$-type, called double EPW quartics, that admit a polarisation of Beauville degree~$4$ and an anti-symplectic involution $\iota\colon Y\to Y$ with invariant lattice $H^2(Y,\IZ)^\iota$ isomorphic to $U(2)$.
    Each double EPW quartic admits two Lagrangian fibrations $Y\to\IP^2$.
    The general double EPW quartic $Y$ admits three distinct realisations:
    \begin{enumerate}[label=\alph*.]
        \item as the base of a $\IP^1$-bundle $\alpha\colon F(X)\to Y$, where $F(X)$ is the Hilbert scheme of $(1,1)$-conics in a Verra fourfold $X=Q\cap C(\IP^2\times\IP^2)$, where $Q$ is a general quadric hypersurface. There exists a unique such Verra fourfold associated to each double EPW quartic.
        \item as the moduli space of stable sheaves twisted by a $2$-torsion Brauer class on a polarised $K3$ surface of genus $2$. The general double EPW quartic admits exactly two such associated $K3$ surfaces.
        \item\label{c_2} as a double cover $Y\to D_1^{\overline{A}}$ of a Lagrangian degeneracy locus associated to a unique Lagrangian space $A$. The base $D_1^{\overline{A}}\subset C(\IP^2\times\IP^2)$ is a quartic section in a cone. The projections induce the two Lagrangian fibrations $Y\to\IP^2$.
    \end{enumerate}
    Only construction \ref{c_2} applies to all double EPW quartics.
\end{theorem}

\subsection{Double EPW quartics via conics in Verra fourfolds}\label{sec:construction-verra}

We begin by reviewing the construction of double EPW quartics via conics in Verra fourfolds, as this is the construction we will be using most, in particular in \Cref{sec:incidence,sec:constant-cycle-surface,sec:multiplicativity}.
Proofs for all statements in this section can be found in \cite[Sec. 3]{ikkr17}.

Let $U_1$ and $U_2$ be three-dimensional complex vector spaces and fix volume forms on~$U_1$ and $U_2^\vee$, such that $\w^2U_1=U_1^\vee$ and $\w^2U_2^\vee=U_2$.
Denote the projective cone over the Segre embedding of $\IP(U_1)\times\IP(\w^2U_2)$ by $C(\IP(U_1)\times \IP(\w^2U_2))$.
Now, define a \emph{Verra fourfold} $X$ as the intersection
$$X\coloneq Q\cap C(\IP(U_1)\times\IP(\w^2 U_2))\subset \IP(\IC\oplus U_1\otimes \w^2U_2)$$
of the cone with a quadric hypersurface $Q$. 
The general Verra fourfold is a smooth Fano variety, see \cite[Lem. 2.3]{laterveer-verra}, and a double cover $\pi\colon X\to\IP(U_1)\times\IP(\w^2U_2)$ branched along a $(2,2)$-divisor.
Denote the composition of $\pi$ with the projections $\pr_i$ by $\pi_i$.

Consider the Hilbert scheme of $(1,1)$-conics on $X$,
$$F(X)\coloneq\left\{c\subset X~|~c\text{ is a conic with }\pi_i(c)=L_i\text{ for some lines }L_i\text{ and }i=1,2\right\}.$$
For the rest of this article, unless otherwise stated, any conic in $X$ will be a $(1,1)$-conic. 
We describe how to construct a double EPW quartic from $F(X)$.

We have that 
$$I_{X,2}\coloneq H^0(X,\mathcal{I}_X(2))\cong \IC\oplus U_1\otimes \w^2U_2,$$
which we think of as the space of quadrics in $\IP(\IC\oplus U_1\otimes\w^2U_2)$ containing $X$. 
Define the following morphism
$$\psi\colon F(X)\to\IP(\IC\oplus \w^2U_1\otimes U_2),~c\mapsto H_c\coloneq\left\{\mathfrak{Q}\in I_{X,2}~|~X\cup\langle c\rangle \subset \mathfrak{Q}\right\},$$
where $\langle c\rangle $ is the plane spanned by $c$. 
Notice that this is well-defined, as $H_c$ is a hyperplane in $I_{X,2}$, meaning it defines a point in the dual projective space $\IP(I_{X,2}^\vee)=\IP(\IC\oplus \w^2U_1\otimes U_2)$.

Any $c\in F(X)$ is contained in 
$$c\subset D_{(L,M)}\coloneq \pi^{-1}(L\times M)$$
for the two lines $L$ and $M$ with $\pi_1(c)=L$ and $\pi_2(c)=M$. 
As a double cover of $\IP^1\times\IP^1$ branched along a $(2,2)$-divisor, this is a degree 4 del Pezzo surface, and we have 
$$D_{(L,M)}\subset \IP(\IC\oplus L\otimes M)\subset\IP(\IC\oplus U_1\otimes \w^2U_2),$$
by abuse of notation, where we also denote the two-dimensional subspaces of $U_1$ and~$\w^2U_2$ defining $L$ and $M$ by the same letter. 
As $c\subset D_{(L,M)}$, there is a unique quadric threefold~$\tilde{Q}$ containing $D_{(L,M)}$ and $\langle c\rangle$, i.e.
$$D_{(L,M)}\cup \langle c\rangle \subset \tilde{Q}\subset \IP(\IC\oplus L\otimes M).$$
Any $\tilde{Q}$ obtained in this manner can only be of rank 3 or 4, with the one defined by the general conic being of rank $4$.
There are two pencils of planes on quadric threefolds of rank $4$, which coincide for those of rank $3$. 
Intersecting any plane contained in $\tilde{Q}$ with~$D_{(L,M)}$ defines a $(1,1)$-conic in $X$. 
The morphism $\psi$ contracts exactly the conics obtained from the same $\tilde{Q}$. 
The involution on $X$ restricts to an involution on~$D_{(L,M)}$, where it interchanges the two pencils of conics obtained from the two pencils of planes in $\tilde{Q}$. 
Thus, the general fibre of $\psi$ consists of two copies of $\IP^1$, which are interchanged by the involution. 
Denote the image of $\psi$ by $D$, which is a singular Calabi--Yau fourfold and a quartic hypersurface in $C(\IP(\w^2 U_1\otimes \IP(U_2)))\subset\IP(\IC\oplus \w^2 U_1\otimes U_2)$.
Then, $\psi\colon F(X)\to D$ factors through a $\IP^1$-bundle
$$\alpha\colon F(X)\to Y,$$
where $Y$ is a double EPW quartic and a double cover of $D$. 
The double EPW quartic is a hyperkähler fourfold and is of $K3^{[2]}$-type.
The family of Verra fourfolds is $19$-dimensional, and so is the family of double EPW quartics.
Interchanging the pencils of conics contained in a $\tilde{Q}$ defines an anti-symplectic involution $\iota$ on $Y$. 
For more details on $D$, see \Cref{sec:construction-lagrangian}, where it is denoted $D_1^{\overline{A}}$.

\subsection{Double EPW quartics via Lagrangian degeneracy loci and their two Lagrangian fibrations}\label{sec:construction-lagrangian}

In this section we review the construction of double EPW quartics via Lagrangian degeneracy loci, following \cite[Sec. 2]{ikkr17}.
The existence of the Lagrangian fibrations and the fact that $Y$ is a double cover of a quartic section in a cone will be crucial to the proof of the Beauville--Voisin--Franchetta conjecture in \Cref{sec:BV}.
While we will not use the details of the construction, we present a summary for the sake of completeness.

Let $V$ be a $6$-dimensional complex vector space and fix a volume form on it, which induces a symplectic form on $\w^3 V$ via the wedge product. 
Denote the space of Lagrangian subspaces of $\w^3 V$ by $\mathrm{LG}(10, \w^3 V)$ and define the Lagrangian subspace $T_U=\w^2 U\w V$ for any $U\in\Gr(3,V)$. Let
\begin{gather*}
    \Sigma'\coloneq \left\{A\in\mathrm{LG}(10,\w^3 V)~|~\IP(A)\cap \Gr(3, V)=\{U\}\text{ for some }U\in\Gr(3,V)\right\},\\
    \Sigma''\coloneq \left\{A\in\mathrm{LG}(10,\w^3 V)~|~\dim(A\cap T_U)\leq 1\text{ for all }U\in\Gr(3,V)\right\},\\
    \Sigma\coloneq \Sigma'\cap\Sigma''.
\end{gather*}
This is a $19$-dimensional family. Consider $A\in\Sigma$ and let $U_1$ be the unique element in~$\IP(A)\cap \Gr(3,V)$. 
Define
$$C_{U_1}\coloneq\IP(T_U)\cap \Gr(3,V)$$
and denote the following quotients by $\overline{A}\coloneq A/\wedge^3 U_1$ and $\overline{T}_U\coloneq T_U/\wedge^3 U_1$. Finally, define
$$D_k^{\overline{A}}\coloneq \left\{[U]\in C_{U_1}~|~\dim(\overline{T}_U\cap \overline{A})\geq k\right\}.$$
It is proven in \cite[Sec. 2]{ikkr17} that $C_{U_1}$ is a cone with vertex $[U_1]$ over the Segre embedding of $\IP(\w^2 U_1)\times\IP(V/U_1)$ and that $D_1^{\overline{A}}$ is a quartic hypersurface in $C_{U_1}$, singular along the surface $D_2^{\overline{A}}$.
Furthermore, it is proven that $D_1^{\overline{A}}$ is integral. 

A natural desingularisation of $D_1^{\overline{A}}$ can be constructed as 
$$\tilde{D}_1^{\overline{A}}\coloneq \left\{([U],[\omega])\in C_{U_1}\times \Gr(1,\overline{A})~|~\omega\subset \overline{T}_U\right\}.$$
Consider also 
$$\tilde{\mathbb{D}}_1^{\overline{A}}\coloneq\left\{(L,\omega)\in\mathrm{LG}(9,(\w^3 U_1)^\perp /(\w^3 U_1))\times \Gr(1,\overline{A})~|~\omega\subset L\right\}$$
and the embedding
$$g\colon C_{U_1}\to\mathrm{LG}(9,(\w^3 U_1)^\perp/(\w^3 U_1)),~U\mapsto \overline{T}_U.$$
Then, we obtain the following commutative diagram. 
\begin{equation*}
    \xymatrix{
    \tilde{D}_1^{\overline{A}} \ar[r]^\phi\ar[d]
    &D_1^{\overline{A}}\ar[r]\ar[d]
    &C_{U_1}\ar[d]^g\\
    \tilde{\mathbb{D}}_1^{\overline{A}}\ar[r]
    &\mathbb{D}_1^{\overline{A}} \ar[r]
    &\mathrm{LG}\bigl(9,(\wedge^3 U_1)^\perp/(\wedge^3 U_1)\bigr)\\
    }
\end{equation*}
It is then shown in \cite{ikkr17} that the exceptional divisor $E$ in $\tilde{D}_1^{\overline{A}}$ is even, meaning that there is a double cover $\tilde{f}\colon \tilde{Y}_A\to \tilde{D}_1^{\overline{A}}$ branched along $E$. 
Furthermore, it is proven that~$\tilde{f}^{-1}(E)$ is contracted by a birational morphism induced by some multiple of $\tilde{f}^\ast\phi^\ast H$, where $H$ is the hyperplane class on $\mathrm{LG}(9,(\wedge^3 U_1)^\perp/(\wedge^3 U_1))$ pulled back to $D_1^{\overline{A}}$.
The image of this morphism is smooth and induces a double cover 
$$Y_A\to D_1^{\overline{A}},$$
branched along $D_2^{\overline{A}}$.
This $Y_A$ is a double EPW quartic.
It is shown in \cite{ikkr17} that it has trivial first Chern class and that there is an $18$-dimensional subset of $\Sigma$ for which~$Y_A$ is birational to the Hilbert scheme of two points on a $K3$ surface, meaning that it is a hyperkähler fourfold of $K3^{[2]}$-type. 

We see again that the family of double EPW quartics is $19$-dimensional and that they are double covers of singular quartic hypersurfaces $D_1^{\overline{A}}$ in a cone $C(\IP^2\times\IP^2)$. 
In particular, the variety $D$ in \Cref{sec:construction-verra} is exactly $D_1^{\overline{A}}$ for a unique $A\in\Sigma$ associated to the Verra fourfold.

Finally, the compositions of the double cover $Y\to D_1^{\overline{A}}$ with the two projections of~$C_{U_1}$ onto $\IP^2$ define two Lagrangian fibrations on $Y$.

\subsection{Double EPW quartics as moduli spaces of twisted sheaves on K3 surfaces}\label{sec:construction-moduli}

In this section we review the construction of double EPW quartics as moduli spaces of twisted sheaves on K3 surfaces.
While we will not use it in this article, we describe the construction for the sake of completeness.

Double EPW quartics were first realised as moduli spaces of twisted sheaves on $K3$ surfaces in \cite[Sec. 4]{ikkr17} via lattice-theoretic methods by constructing an embedding of $H^2(Y,\IZ)$ into the Hodge structure associated to the twisted $K3$ surface and finding a vector with square $2$ in the orthogonal complement.
Here, we follow the more geometric approach of \cite[Sec. 5]{CamereKapustkaMongardi2018}.

Let $X=Q\cap C(\IP^2\times\IP^2)$ be a general Verra fourfold.
The two projections $\pi_i\colon X\to~\IP^2$ are quadric fibrations, whose discriminant loci are smooth sextic curves $C_i\subset \IP^2$. 
Let~$S_i$ be the double covers of $\IP^2$ branched along $C_i$. 
The quadric bundles $\pi_i\colon X\to\IP^2$ induce~$\IP^1$-fibrations on $S_i$, defining two Brauer classes $\beta_i\in\mathrm{Br}(S_i)$ of order $2$.
We work with $\pi_1$ and $S_1$, as the other case is analogous. For the remainder of this section, we denote $\Ox_X(a,b)=\pi_1^\ast\Ox_{\IP^2}(a,b)$.

We know from \cite[Thm. 4.2]{kuznetsov} that there is a semi-orthogonal decomposition
$$D^b(X)=\left\langle\Phi(D^b(\IP^2,\mathcal{B}_0)),\Ox_X(-1,0),\Ox_X,\Ox_X(1,0),\Ox_X(0,1),\Ox_X(1,1),\Ox_X(2,1)\right\rangle,$$
where $\Phi\colon D^b(\IP^2,\mathcal{B}_0)\to D^b(X)$ is a fully faithful functor and $\mathcal{B}_0$ is the sheaf of even parts of the Clifford algebra associated to $\pi_1$, as in \cite{kuznetsov}.
Furthermore, it is shown in \cite[Lem. 4.2]{kuznetsov2} that the double cover $f\colon S_1\to\IP^2$ induces an equivalence
$$f_\ast\colon D^b(S_1,\beta_1)\to D^b(\IP^2,\mathcal{B}_0).$$
Let $\Gamma$ be the inverse functor and let $\pr\colon D^b(X)\to \Phi(D^b(\IP^2,\mathcal{B}_0))$ be the projection functor. 

Consider the map 
$$F(X)\to\mathcal{M}_{(0,H,0)}(S_1,\beta_1),$$
where $H$ is the polarisation on $S_1$, given by the composition of $\Gamma$ with 
$$c\mapsto \pr(\mathcal{I}_{c/X}(1)).$$
It is shown in \cite[Thm. 5.1]{CamereKapustkaMongardi2018} that this map is well-defined and it factors through~$\alpha\colon F(X)\to Y$ and induces an isomorphism
$$Y\cong \mathcal{M}_{(0,H,0)}(S_1,\beta_1).$$
Furthermore, this can be done with both $K3$ surfaces $S_1$ and $S_2$ associated to $X$, and hence we also have $Y\cong \mathcal{M}_{(0,H',0)}(S_2,\beta_2)$.

\section{The Franchetta property}\label{sec:franchetta}
In this section we discuss the Franchetta property and prove that it holds for the family of double EPW quartics and cubes of Verra fourfolds.
\begin{definition}\label{def:gen-defined}
    Let $\mathcal{X}\to B$ be a family of smooth varieties over a smooth base $B$. Let~$b\in B$ be a point. 
    Following \cite{flv21}, we say that a cycle $\gamma\in\chow^\ast(\mathcal{X}_b)$ is \emph{generically defined} if it lies in the image of the pullback map from $\mathcal{X}$, i.e.
    $$\gamma\in\im\left(\chow^\ast(\mathcal{X})\to\chow^\ast(\mathcal{X}_b)\right).$$
    We denote the subring of generically defined cycles over $B$ by $\gdch_B^\ast(\mathcal{X}_b)$.
\end{definition}
Notice that the definition of generically defined cycles depends on the family $\mathcal{X}\to B$. 
It will usually be clear over which base we are working, so we will often omit the subscript~$B$.
\begin{definition}\label{def:franchetta-property}
    Let $\mathcal{X}\to B$ be a family of smooth varieties over a smooth base $B$.
    We say that the family $\mathcal{X}\to B$ satisfies the Franchetta property if for any closed point, or equivalently for the very general point, $b\in B$, we have that the cycle class map restricted to the subring of generically defined cycles,
    $$\mathrm{cl}\colon\gdch_B^\ast(\mathcal{X}_b)\to H^\ast(\mathcal{X}_b,\IQ),$$
    is injective.
\end{definition}

Beauville and Voisin proved in \cite{BV} the existence of a special zero-cycle on a $K3$ surface $S$ satisfying $D\cdot D'=\deg(D\cdot D') o_S$ for any two divisors $D,D'$ and that $c_2(S)=~24 o_S$.
The class of any point lying on a rational curve in $S$ is $o_S$.
O'Grady noticed that this implies that the Beauville--Voisin class $o_S$ is generically defined and asked in \cite{og2} whether the Franchetta property holds for the family of $K3$ surfaces. 
This was extended to hyperkähler varieties in \cite{flv19}, and it was conjectured in \cite{bvf} that the Generalised Franchetta conjecture holds.

\begin{conjecture}[\cite{bvf}]\label{conj:franchetta}
Let $\mathcal{F}$ be the moduli stack of a locally complete family of polarised hyperkähler varieties and let $\mathcal{X}\to\mathcal{F}$ be the universal family. 
\begin{itemize}[label={}]
\item (Franchetta Conjecture): The family $\mathcal{X}\to\mathcal{F}$ satisfies the Franchetta property.
\item (Generalised Franchetta Conjecture): The family $\mathcal{X}^{n/\mathcal{F}}\to\mathcal{F}$ satisfies the Franchetta property for any $n \geq 1$.
\end{itemize}
\end{conjecture}
Since we work exclusively with varieties, we refer to \cite{flv19} for more details on the fact that the conjecture is stated in terms of stacks. 
\begin{remark}
    Unfortunately, what we call the Franchetta conjecture above is called the generalised Franchetta conjecture in \cite{og2}, \cite{pavic-shen-yin}, \cite{flv19}, and \cite{flv21}. 
    We have decided to follow the convention used in \cite{bvf} and distinguish between the two, as there is less evidence for what we call the generalised Franchetta conjecture and because it interacts differently with conjectures such as the Beauville--Voisin conjecture, see \Cref{sec:BV}.
\end{remark}
The Franchetta property was first proved for the family of $K3$ surfaces of genera~$g\leq~10$ and $g=12,13,16,18,20$ in \cite{pavic-shen-yin}, and more recently for $g=14$ \cite{fu-laterveer} and $g=11$ \cite{lu2025ogradysgeneralizedfranchettaconjecture}.
This has been extended to squares of $K3$ surfaces of genus $g\leq 10$ or~$g=12$, cubes of $K3$ surfaces of genus $2,3,4,5$, and fifth powers of quartic $K3$ surfaces in \cite{flv19}.
Furthermore, the Franchetta property is known for the Fano variety of lines on cubic fourfolds and its self-product \cite{flv19}, fourth powers of cubic fourfolds, and LLSS eightfolds \cite{flv21}. 
We can now add double EPW quartics and Verra fourfolds and their Hilbert schemes of $(1,1)$-conics to this list.

\begin{theorem}\label{thm:franchetta}
    Let $\mathcal{Y}\to \mathcal{U}$ be the family of double EPW quartics. It satisfies the Franchetta property, meaning that for any point $b\in\mathcal{U}$, the cycle class map restricted to generically defined cycles,
    $$\mathrm{cl}\colon\gdch^\ast(\mathcal{Y}_b)\to H^\ast(\mathcal{Y}_b,\IQ),$$
    is injective.
    The universal family of self-products of Verra fourfolds $\mathcal{X}^{n/\mathcal{U}}\to~\mathcal{U}$ also satisfies the Franchetta property for $n\leq 3$.
    So does the family of Hilbert schemes of~$(1,1)$-conics $F(\mathcal{X})\to\mathcal{U}$.
\end{theorem}
\begin{proof}
    The construction of double EPW quartics via moduli of stable objects, see \Cref{sec:construction-moduli}, can be done in families, yielding a dominant rational morphism $\mathcal{F}_4\dashrightarrow \mathcal{U}$ from the moduli space of genus-$2$ $K3$ surfaces $\mathcal{F}_4$. 
    We know from the construction, see \cite[Rmk. 4.3]{ikkr17} for more details, that this is generically finite of degree $2$. 
    Then, \cite[Thm. 1.5]{flv21} describes the Chow motive of moduli spaces of sheaves on a $K3$ surface as a direct summand of $\bigoplus_{i=1}^r\mathfrak{h}(S^2)(l_i)$, where $r\in\IN$ and $l_i\in\IZ$.
    We refer to \cite{flv21} for notation.
    Then, it is shown in \cite[Thm. 1.5]{flv19} that the family $\mathcal{S}\times_{\mathcal{F}_4}\mathcal{S}\to\mathcal{F}_4$ of self-products of $K3$ surfaces of genus $2$ satisfies the Franchetta property. 
    If the pullback of a family along a dominant morphism satisfies the Franchetta property, so does the original family, see \cite[Rmk. 2.6]{flv19}, completing the proof of the first statement.

    The argument for the second statement is similar. It is known from \cite{laterveer-verra}, see also \Cref{rmk:verra-fix}, that there is an isomorphism of Chow motives $\mathfrak{h}(X)\cong\mathfrak{h}(S)(1)\oplus\bigoplus_{i=0}^4\mathds{1}(i)$, where $S$ is a $K3$ surface associated to the double EPW quartic. 
    This can be done in families and hence it follows from \cite[Thm. 1.5]{flv19} that the family of triple products of genus-$2$ $K3$ surfaces satisfies the Franchetta property. 

    We know from \Cref{thm:construction} that $F(X)$ is a $\IP^1$-bundle over the double EPW quartic~$Y$.
    The projective bundle formula for motives and the Franchetta property for $\mathcal{Y}\to\mathcal{U}$ then imply the Franchetta property for $F(\mathcal{X})\to\mathcal{U}$.
\end{proof}

\begin{remark}\label{rmk:non-complete}
    In general, one only expects locally complete families of hyperkähler varieties to satisfy the Franchetta property, which moduli spaces of stable objects on $K3$ surfaces aren't.
    Still, their motives are summands of the motives of powers of $K3$ surfaces, which are expected to satisfy the generalised Franchetta conjecture.
    Hence, we would expect these families to also satisfy the Franchetta property, even though they do not form a locally complete family. 
\end{remark}

This implies that all generically defined cycles on $Y$ multiply to a multiple of $c_4(Y)$ in~$\chow_0(Y)$. 
Hence, we define the \emph{Beauville--Voisin class} $o_Y$ as $\frac{c_4(Y)}{\mathrm{deg} c_4(Y)}$. 
This is well-defined since it is known that $\mathrm{deg} c_4(S^{[2]})=324$ for any $K3$ surface $S$ and that this number is deformation invariant, meaning that $\mathrm{deg} c_4(Y)=324$.
We will prove in \Cref{sec:constant-cycle-surface} that $o_Y$ is not only a $0$-cycle of degree $1$, but that it is represented by a point in $Y$.
Now, let $h_1$ and $h_2$ be the two divisors pulled back from the Lagrangian fibrations $Y\to\IP^2$. 
These are the only divisors on the very general double EPW quartic.
Similarly, the only Hodge classes in degree $4$ on the very general double EPW quartic are $h_1^2, h_2^2, h_1h_2,$ and $c_2(Y)$. 
From the Hard Lefschetz theorem we know that the only Hodge classes in degree $6$ on the very general double EPW quartic are $h_1^2h_2$ and $h_1h_2^2$. 
Thus, we have that 
\begin{align*}
    \gdch^1(Y)&=\langle h_1, h_2\rangle\\
    \gdch^2(Y)&=\langle h_1^2, h_2^2, h_1h_2, c_2(Y)\rangle\\
    \gdch^3(Y)&=\langle h_1^2h_2, h_1h_2^2\rangle\\
    \gdch^4(Y)&=\langle o_Y\rangle 
\end{align*}
In particular, generically defined cycles on $Y$ are $\iota$-invariant.
\begin{remark}\label{rmk:prim-div}
    Denote the space of primitive divisors by $H^2(Y,\IQ)_{\mathrm{pr}}=\langle h_1,h_2\rangle^\perp$, relative to the Beauville--Bogomolov form. The subring of $\chow^\ast(Y)$ generated by primitive divisors also maps injectively into cohomology via the cycle class map.
    Indeed, using the twisted incidence correspondence from \cite{twistedSYZ}, the proof of \cite[Thm. C.8]{vial22} applies verbatim in the case of moduli of stable objects on twisted $K3$ surfaces.
    Hence, \cite[Cor. C.11]{vial22} also holds for moduli of stable objects on twisted $K3$ surfaces.
    This shows that $Y$ admits a surface decomposition in the sense of \cite[Def. 1.2]{voisin22} by the self-product of a $K3$ surface. 
    Then, applying \cite[Thm. 1.7]{voisin22} grants the result.
\end{remark}

\section{The incidence variety}\label{sec:incidence}

Voisin studied the incidence variety of the Fano variety of lines on a cubic fourfold and proved a relation involving it to prove the Beauville--Voisin conjecture for the Fano variety \cite{voisin08}.
A similar idea was used to prove an analogue to \Cref{intro:chow-iso} in \cite{zhang24} for the incidence variety of two conics on a general Gushel--Mukai fourfold and the associated double EPW sextic.
The relation Zhang obtained was then pushed forward to the double EPW sextic in \cite{laterveer23}, see \Cref{rmk:error}, and used to prove the Beauville--Voisin conjecture for double EPW sextics.

The aim of this rather technical section is to study the incidence variety of $(1,1)$-conics on a Verra fourfold $X$ and to prove a relation involving it in the Chow ring of~$F(X)\times F(X)$, before pushing it forward to $Y\times Y$. 

\subsection{\texorpdfstring{The incidence variety in $F(X)$}{The incidence variety in F(X)}}\label{sec:incidence1}

Recall from \Cref{sec:construction-verra} that $F(X)$ is the Hilbert scheme of $(1,1)$-conics on $X$ and that there is a $\IP^1$-bundle $\alpha\colon F(X)\to Y$, where $Y$ is the double EPW quartic associated to $X$. 
Unless otherwise stated, all conics in $X$ will be $(1,1)$-conics.
Recall also the double cover $\pi\colon X\to\IP(U_1)\times\IP(\w^2U_2)$.

\begin{lemma}\label{lemma:injective}
Let $X$ be a general Verra fourfold. 
The morphism 
$$F(X)\to\Gr(3,\IC\oplus U_1\otimes\w^2U_2),~c\mapsto \langle c\rangle,$$ 
is injective, where $\langle c\rangle$ is the plane spanned by $c$.
Furthermore, if two conics $c$ and $c'$ have no common component and the planes they generate intersect in a line, i.e. $\langle c\rangle\cap \langle c'\rangle =\ell$, we have $\iota(\alpha(c))=\alpha(c')$. 
\end{lemma}
\begin{proof}
Recall from \Cref{sec:construction-verra} the morphism 
$$\psi\colon F(X)\to\IP(\IC\oplus\w^2 U_1\otimes U_2),~c\mapsto H_c=\left\{\mathfrak{Q}\in I_{X,2}~|~X\cup \langle c\rangle \subset \mathfrak{Q}\right\}$$
used in the construction of double EPW quartics associated to Verra fourfolds.

Let $c,c'\in F(X)$ and assume that $\langle c\rangle = \langle c'\rangle$.
Then, we have that $\psi(c)=\psi(c')$ and the conics lie in the same quadric threefold $\tilde{Q}$ and are uniquely determined by the plane in $\tilde{Q}$ that they generate, hence $c=c'$. 

Now assume instead that $\langle c\rangle\cap\langle c'\rangle = \ell$. 
As $X\cap\langle c\rangle= c$, $\ell$ is not contained in $X$. 
Otherwise, it would be contained in $c$ and $c'$ and hence they would have a common component. 
Then, any quadric $\mathfrak{Q}$ that contains $X\cup\ell$ also contains $\langle c\rangle$ and $\langle c'\rangle$. 
Indeed, we have $$\ell\cup c\subset \langle c\rangle\cap\mathfrak{Q},$$
and counting degrees, this cannot be a proper intersection, meaning that $\langle c\rangle\subset \mathfrak{Q}$. 
The same argument works for $c'$, meaning that $\psi(c)=\psi(c')$, which shows that they are contained in the same quadric threefold $\tilde{Q}$ and obtained from planes contained in it. 
In the case where the quadric threefold $\tilde{Q}$ is of rank $3$, there is a single pencil of planes in~$\tilde{Q}$, meaning that $\alpha(c)=\alpha(c')$ and we also know that in this case $\alpha(c)$ is fixed by $\iota$. 
In the case where $\tilde{Q}$ is of rank $4$, there are two pencils of planes in $\tilde{Q}$, which are exchanged by $\iota$. 
As planes lying on the same pencil only intersect in the vertex of $\tilde{Q}$, two conics on the same pencil can't intersect, meaning that the two conics are associated to planes lying on different pencils, i.e. $\iota(\alpha(c))=\alpha(c')$.
\end{proof}
Given the above, we can consider the closed embedding
$$F(X)\to \Gr(3,\IC\oplus U_1\otimes\w^2 U_2),~c\mapsto \langle c\rangle.$$
Pulling back the tautological bundle $\mathcal{S}$ to $F(X)$, we have the following diagram, where~$P$ is the universal $(1,1)$-conic and $i$ is a closed immersion. 
\begin{equation*}
\xymatrix{
& P\ar[dl]^p\ar[r]^q\ar[d]^i & X\ar[d]\\
F(X) & \IP(\mathcal{S}|_{F(X)})\ar[l]^{p'}\ar[r]^{q'} & \IP(\IC\oplus U_1\otimes \w^2 U_2)
}
\end{equation*}
We also define the incidence correspondence
$$I_F\coloneq\left\{(c,c')\in F(X)\times F(X)~|~c\cap c'\neq\varnothing\right\}.$$

\begin{lemma}\label{lemma:normal-bundle}
Let $c$ be a smooth $(1,1)$-conic on a general Verra fourfold $X$. 
Then its normal bundle splits as
$$N_{c/X}=\Ox_c(1)\oplus \Ox_c(1)\oplus\Ox_c.$$
Furthermore, the morphism $q\colon P\to X$ is flat. 
In particular, we have that 
$$I_F=(p\times p)_\ast (q\times q)^\ast\Delta_X\in \chow^2(F(X)\times F(X)).$$
\end{lemma}
\begin{proof}
We have the following exact sequence of normal bundles
$$0\to N_{c/D_{(L,M)}}\to N_{c/X}\to N_{D_{(L,M)}/X}|_c\to 0. $$
Since $D_{(L,M)}$ is a complete intersection, we also have $$N_{D_{(L,M)}/X}|_c=\pi^\ast(\Ox(1,0)\oplus\Ox(0,1))|_c=\Ox_c(1)\oplus\Ox_c(1).$$
Furthermore, we have $$N_{c/D_{(L,M)}}=\omega_{D_{(L,M)}}|_c^\vee\otimes\omega_c$$ and as $D_{(L,M)}\to L\times M$ is a degree 2 ramified cover, we have $$\omega_{D_{(L,M)}}=\pi^\ast(\Ox(-2,-2))\otimes\Ox(R)=\pi^\ast\Ox(-1,-1),$$ where $2R$ is the ramification locus. 
Since $\pi^\ast\Ox(-1,-1)$ restricts to $\Ox_c(-2)$ on $c$, it follows that~$N_{c/D_{(L, M)}}=\Ox_c$. 
Finally, $c\cong \IP^1$, since $c$ is smooth, and it follows that the first sequence splits, and we obtain the desired decomposition of the normal bundle.

Being a divisor in the smooth variety $\IP(\mathcal{S}|_{F(X)})$, $P$ is Cohen--Macaulay, meaning that, by miracle flatness, it suffices to check that the fibres of $q$ are equi-dimensional. 
There is a smooth conic $c$ in every fibre $q^{-1}(x)$.
We may identify the fibre $p^{-1}(c)$ with $c$ via $q$ and see that $N_{c/P}=T_c F(X)\otimes\Ox_c$. 
The following diagram commutes.
\begin{equation*}
\xymatrix{
0\ar[r] & T_c\ar[r]\ar[d]^\cong & T_P|_c\ar[d]^{Dq|_c}\ar[r] & (T_{F(X)})_c\otimes\Ox_c\ar[d]^{\mathrm{ev}}\ar[r]& 0\\
0\ar[r]& T_c\ar[r]& T_X|_c\ar[r]& N_{c/X}\ar[r]&0
}
\end{equation*}
Then, $T_{c,x}q^{-1}(x)=\{\sigma\in H^0(N_{c/X})~|~\sigma_x=0\}$. 
Using the description of $N_{c/X}$ above, we see that this is 2-dimensional. 
\end{proof}
Now, we require some more notation. Let $I_0=I_F\setminus (W_1\cup W_2)$, where 
$$W_1 \coloneq \left\{(c,c')~|~c\text{ and }c'\text{ have a common component}\right\}$$
 and 
 $$W_2\coloneq\left\{(c,c')\in F(X)\times F(X)~|~\iota(\alpha(c))=\alpha(c')\right\}.$$
Furthermore, let $$\tilde{I}\coloneq (q\times q)^{-1}(\Delta_X) \text{ and } I'\coloneq (q'\times q')^{-1}(\Delta_{\IP(\IC\oplus U_1\otimes\w^2 U_2)})$$ be the incidence correspondences. 
Also, let $$\tilde{I}_0\coloneq(p\times p)^{-1}(I_0) \text{ and } I'_0\coloneq(p'\times p')^{-1}(I_0).$$
The second part of \Cref{lemma:injective} shows that sending $(c,c')\in I_0$ to the point of intersection of the planes they generate is a well-defined section from $I_0$ to $I_0'$.
\begin{lemma}\label{lemma:regular-embedding}
The inclusion $I_0\subset F(X)\times F(X)\setminus(W_1\cup W_2)$ is a regular embedding.
\end{lemma}
\begin{proof}
This is analogous to \cite[Lem. 3.7]{zhang24}. 
The diagonal $\Delta_{\IP(\IC\oplus U_1\otimes\w^2 U_2)}$ is a regular embedding and $I'$ is obtained from it by base-change. 
Since there exists a section from~$I_0$ to $I'_0$, we apply \cite[B.7.5]{fulton}, which completes the proof.
\end{proof}
We can now compute the self-intersection of $I_F$, obtaining an analogous relation to the case of cubic fourfolds in \cite{voisin08} and of double EPW sextics in \cite{zhang24}.
\begin{definition}\label{def:decomposable}
    Let $T$ and $T'$ be varieties. 
    We say that a cycle $\Gamma\in\chow^\ast(T\times T')$ is \emph{decomposable} if it lies in 
    $$\pr_1^\ast\chow^\ast(T)\cdot\pr_2^\ast\chow^\ast(T').$$
\end{definition}
\begin{proposition}\label{prop:I_F^2}
Let $X$ be a general Verra fourfold. 
Then we have the following relation
$$I_F^2=aW_2+I_F\cdot \Gamma+\Gamma'\text{ in }\chow^4(F(X)\times F(X))$$
for some $a\in\IQ$ and some decomposable and generically defined cycles $\Gamma,\Gamma'$.
\end{proposition}

\begin{remark}\label{rmk:non-zero}
    Notice that we haven't stated whether the coefficient $a$ is non-zero.
    This is indeed the case, but we will postpone the proof of this to \Cref{lemma:non-zero} for the sake of readability.
    The reader will verify that only \Cref{cor:proj-i} requires $a\neq 0$ and that the proof of \Cref{lemma:non-zero} does not rely on it. 
\end{remark}

\begin{proof}[Proof of \Cref{prop:I_F^2}]
We shorten $F(X)$ to $F$ and denote the restrictions of $p\times p$ and~$q\times q$ to $\tilde{I}_0$ by $p_0$ and $q_0$, respectively. 
Since $W_1$ is of dimension $5$, using the localisation sequence, it is sufficient to prove that $I_0^2=I_0\cdot \Gamma+\Gamma'$ for $\Gamma,\Gamma'$ as in the statement. 
By \Cref{lemma:regular-embedding}, $I_0\subset F(X)\times F(X)\setminus(W_1\cup W_2)$ is a regular embedding, meaning that
$$I_0^2=c_2(N_{I_0/F\times F}).$$
Furthermore, $p_0$ is an isomorphism from $\tilde{I}_0\subset P\times P$ to $I_0\subset F\times F$, meaning that we have the following exact sequence.
$$0\to \pr_1^\ast T_{P/F}\oplus\pr_2^\ast T_{P/F}\to N_{\tilde{I}_0/P\times P}\to p_0^\ast(N_{I_0/F\times F})\to 0 $$
Hence, the Chern classes of $p_0^\ast N_{I_0/F\times F}$ are polynomials in those of $\pr_i^\ast T_{P/F}$ and $N_{\tilde{I}_0/P\times P}$. 
Since $\tilde{I}_0=q_0^{-1}(\Delta_X)$, we have $c(N_{\tilde{I}_0/P\times P})=q_0^\ast c(T_X)$. 
We prove now that these are decomposable.

Denote by $H$ and $G$ the divisors associated to $\Ox(1,1)$ and $\Ox(1,0)$ on $\IP(U_1)\times \IP(\w^2 U_2)$, which we shorten to $\IP^2\times \IP^2$. 
Since removing the vertex of $C(\IP^2\times\IP^2)$ results in the total space of a line bundle over $\IP^2\times\IP^2$, the localisation sequence shows that $$\chow^\ast(C(\IP^2\times\IP^2))\cong\chow^\ast(\IP^2\times\IP^2)$$ and that it is generated by $H$ and $G$. 
Using the conormal sequence, it follows that the first two Chern classes of $T_X$ are polynomials in $G$ and $H$.

We can also compute
$$c(T_{P/F})=\frac{c(T_{\IP(\mathcal{S})/F})|_P}{c(N_{P/\IP(\mathcal{S})})}=\frac{c(p^\ast\mathcal{S}\otimes\Ox_{\IP(\mathcal{S})}(1))|_P}{c(\Ox_P(P))},$$
where we also denote the divisor class of $P$ in $\IP(\mathcal{S})$ by the same letter. 
More precisely, we have $P=2c_1(\Ox_{\IP(\mathcal{S})}(1))+p^\ast D$ for some $D\in\chow^1(F)$ and $c_1(\Ox_{\IP(\mathcal{S})})|_P=q^\ast H$. 
Hence,~$c(T_{P/F})$ is a polynomial in $q^\ast H$ and $p^\ast\chow^\ast(F)$.

Thus, we obtain 
$$I_0^2=p_{0\ast}p^\ast c_2(N_{I_0/F\times F})=p_{0\ast}(P(D, H_i, G_i)),$$
for some polynomial $P(D,H_i, G_i)$ in $D\in p^\ast(\pr_1^\ast\chow^\ast(F)\cdot \pr_2^\ast\chow^\ast(F))$, $H_i=\pr_i^\ast q^\ast H$, and $G_i=\pr_i^\ast q^\ast G$.
We can further rewrite this as
$$I_0^2=(p\times p)_\ast(P(D,H_i,G_i)\cdot \tilde{I})|_{F\times F\setminus(W_1\cup W_2)}=(p\times p)_\ast(P(D,H_i,G_i)\cdot(q\times q)^\ast\Delta_X)|_{F\times F\setminus(W_1\cup W_2)}.$$
Decompose the polynomial as $$P(D,H_i,G_i)=L_1\cdot A+L_2\cdot B+Q,$$ where $L_1$ and $L_2$ are linear in $H_1,G_1,$ and $H_2,G_2$, respectively, $A$ and $B$ are linear, and~$Q\in p^\ast(\pr_1^\ast\chow^\ast(F)\cdot\pr_2^\ast\chow^\ast(F))$. 
It follows from the projection formula that
$$(p\times p)_\ast(Q\cdot \tilde{I})=\Gamma\cdot I,$$
where $\Gamma$ is decomposable. 
As $A$ and $B$ are decomposable, it suffices to prove that both~$H_i\cdot (q\times q)^\ast\Delta_X$ and $G_i\cdot(q\times q)^\ast\Delta_X$ are decomposable.

Before continuing, notice that for any other variety $T$, the exterior product 
$$\bigoplus_{i+j=k}\chow^i(T)\otimes\chow^j(C(\IP^2\times\IP^2))\to\chow^k(T\times C(\IP^2\times\IP^2))$$
is surjective. 
Indeed, consider the following diagram, where $U$ is the complement of the vertex $v$ of $C(\IP^2\times\IP^2)$.
\begin{equation*}
\xymatrix{
\chow^\ast(T)\ar[d]\ar[r]&\chow^\ast(T)\otimes\chow^\ast(C(\IP^2\times\IP^2))\ar[d]\ar[r]&\chow^\ast(T)\otimes\chow^\ast(U)\ar[d]\ar[r]& 0\\
\chow(T\times\{v\})\ar[r]&\chow^\ast(T\times C(\IP^2\times\IP^2))\ar[r]& \chow^\ast(T\times U)\ar[r] & 0
}
\end{equation*}
The rows are exact, and the first and third vertical maps are isomorphisms, since $U$ is an $\mathbb{A}^1$-bundle over $\IP^2\times\IP^2$. 
Hence, by the Four lemma, the second vertical map is surjective. 
Now, consider $$j\colon X\times X\to C(\IP^2\times\IP^2)\times X$$
Then, the argument above shows that $j_\ast\Delta_X\in\pr_1^\ast\chow^\ast(C(\IP^2\times\IP^2))\cdot \pr_2^\ast\chow^\ast(X)$, implying that 
$$2\pr_1^\ast H\cdot\Delta_X=j^\ast j_\ast\Delta_X\in\pr_1^\ast\chow^\ast(X)\cdot\pr_2^\ast\chow^\ast(X).$$
Finally, $G$ is represented by a subvariety $V$ that becomes an $\mathbb{A}^1$-bundle over $\IP^1\times\IP^2$ after removing the vertex. 
By an analogous argument as above, every cycle on $V\times V$ is decomposable. 
Thus, as $\pr_i^\ast G\cdot\Delta_X$ is the push-forward of $\Delta_V$, it is also decomposable. 
We obtain that
$$(p\times p)_\ast((L_1\cdot A+L_2\cdot B)\cdot(q\times q)^\ast\Delta_X)$$
is of the form $\Gamma'$ as in the statement of the proposition. 
The fact that the cycles $\Gamma$ and $\Gamma'$ are generically defined follows directly, as they are all given by Chern class computations involving only the polarisations, tangent bundles, and normal bundles of varieties that can be constructed over the family $\mathcal{U}$.
\end{proof}

\subsection{\texorpdfstring{Pushing the incidence variety forward to $Y$}{Pushing the incidence variety forward to Y}}\label{sec:incidence2}
We continue with pushing the relation forward to $Y\times Y$ to obtain a relation on the double EPW quartic. We prove versions of \cite[2.27, 2.28, and 3.2]{laterveer23}. Obtaining them is slightly simpler in our case, as the family of double EPW quartics satisfies the Franchetta property and we do not have to pass via algebraic equivalence, as is done in \emph{loc. cit.}

\begin{remark}\label{remark:fibre-of-I}
The general fibre of the restriction of $\alpha\times\alpha$ to $I_F$ is $1$-dimensional.
Indeed, the fibre of $\alpha\times\alpha$ restricted to $I$ is positive-dimensional, as otherwise $I_F\to Y\times Y$ would be surjective, which is a contradiction. 
Assume now that for a general $(c,c')\in I_F$, we have $\{c\}\times\alpha^{-1}(\alpha(c'))\subset I_F$. 
For a general $c'$, any other conic $\tilde{c}$ with $\alpha(\tilde{c})=\alpha(c')$ doesn't intersect $c'$. 
This can be seen from the description of the fibres of $\alpha$. 
Hence, the points of intersection of $c$ with the conics in the fibre of $\alpha(c')$ are all distinct, meaning that all points of $c$ are points of intersection with $c'$ and hence are contained in the same $D_{(L,M)}$ as $c'$.
Were this true for the general $(c,c')$, the dimension of $I$ would be~$5$ and not $8$, which is a contradiction. 
Hence, the fibre of $\alpha\times\alpha$ restricted to $I$ is one-dimensional and is a divisor in the fibre $\IP^1\times\IP^1$ of positive degree in both factors. 
Thus,~$dI_F=h_{F\times F}\cdot(\alpha\times\alpha)^\ast \tilde{I}$ and $d'\tilde{I}=(\alpha\times\alpha)_\ast(I_F\cdot h_{F\times F})$ for some $d,d'\in\IQ$, where~$\tilde{I}$ is the class of the image of $I$ under $\alpha\times\alpha$.
\end{remark}

Due to this remark, we can't work with the image of the class $I_F$ under $\alpha\times\alpha$, but instead choose an ample $h_F$ on $F$ and work with $$I\coloneq(\alpha\times\alpha)_\ast(I_F\cdot h_F\times h_F).$$

\begin{remark}\label{rmk:error}
    It is claimed in \cite[Rmk. 2.26]{laterveer23} that, with the notation of \emph{loc. cit.}, the entire fibre of $\pi\times\pi$ is contained in $I_F$. 
    However, by the same argument as in \Cref{remark:fibre-of-I}, the general fibre of $\pi\times\pi$ restricted to $I_F$ is $1$-dimensional. 
    As such, the statement of \cite[Rmk. 2.26]{laterveer23} is wrong.
    This has been fixed in an erratum \cite[Sec. 4]{laterveer-erratum}.
    \Cref{pullback-I} below is an adaptation of \cite[Lem. 4.2]{laterveer-erratum}.
\end{remark}

\begin{lemma}\label{pullback-I}
    Let $X$ be a general Verra fourfold.
    The class 
    $$A \coloneq I_F-(\alpha\times\alpha)^\ast I\text{ in }\chow^2(F(X)\times F(X))$$
    is decomposable and generically defined. 
\end{lemma}
\begin{proof}
    We shorten $F(X)$ to $F$.
    It follows directly that $A$ is generically defined, as $I_F$ and $\alpha$ are defined in the family $\mathcal{U}$. 
    Thus, we only need to prove that $A$ is decomposable. 
    To this end, consider the projective bundle formula for $\alpha\times\alpha$, which tells us that
    $$A=(\alpha\times\alpha)^\ast (\gamma)+(h_F\times F)\cdot(\alpha\times\alpha)^\ast(d_1)+(F\times h_F)\cdot(\alpha\times\alpha)^\ast(d_2)+\lambda h_F\times h_F, $$
    where $\lambda\in \IQ$, $d_1, d_2\in\chow^1(Y\times Y)$ and $\gamma\in\chow^2(Y\times Y)$. 
    Since $H^1(Y,\Ox_Y)=0$, \cite[Ex.~III.12.6]{hartshorne} shows that all divisors on $Y\times Y$ are decomposable. 
    This means we only need to prove that $(\alpha\times\alpha)^\ast(\gamma)$ is decomposable. 
    We have
    \begin{align*}
        \gamma=&(\alpha\times\alpha)_\ast(A\cdot h_F\times h_F)+\pr_1^\ast c_1\cdot (\alpha\times\alpha)_\ast(A\cdot \pr_1^\ast h_F)\\
        &+\pr_2^\ast c_1\cdot(\alpha\times\alpha)_\ast(A\cdot\pr_2^\ast h_F)+(c_1\times c_1)\cdot(\alpha\times\alpha)_\ast(A),
    \end{align*}
    where $c_1$ is the first Chern class of the vector bundle underlying the projective bundle~$\alpha$. 
    The first term vanishes by definition, while the other terms are decomposable, as they are all products of divisors.
\end{proof}

We may now push the relation involving $I_F^2$ forward to $Y\times Y$.

\begin{proposition}\label{prop:I^2}
    Let $Y$ be a general double EPW quartic with associated Verra fourfold $X$ and let $I$ be the cycle $(\alpha\times\alpha)_\ast(I_F\cdot h_F\times h_F)$, where $h_F$ is an ample divisor on $F(X)$. 
    Then there is a relation
    $$I^2=a\Gamma_\iota+I\cdot \Gamma + \Gamma' \text{ in }\chow^4(Y\times Y)$$
    where $a\in \IQ$ and $\Gamma$ and $\Gamma'$ are decomposable and generically defined, and $\Gamma_\iota$ is the graph of the involution $\iota$ on $Y$.
\end{proposition}
\begin{proof}
    The statement is analogous to \cite[Prop. 2.27]{laterveer23} and the proof is similar. 
    We use \cref{prop:I_F^2,pullback-I} to obtain
    $$((\alpha\times\alpha)^\ast I+A)\cdot I_F=aW_2+(\alpha\times\alpha)^\ast I\cdot\Gamma+\Gamma\cdot A+\Gamma'.$$
    Intersecting with $h_F\times h_F$ and pushing forward, we use the projection formula to simplify the equation to 
    $$ I^2+(\alpha\times\alpha)_\ast(A\cdot I_F\cdot h_F\times h_F)=a\Gamma_\iota+I\cdot(\alpha\times\alpha)_\ast(\Gamma\cdot h_F\times h_F)+(\alpha\times\alpha)_\ast((\Gamma\cdot A+\Gamma')\cdot h_F\times h_F)$$
    By substituting $I_F=A+(\alpha\times\alpha)^\ast I$ in the second term on the left-hand side, we obtain 
    $$(\alpha\times\alpha)_\ast(A\cdot I_F\cdot h_F\times h_F)=I\cdot (\alpha\times\alpha)_\ast(A\cdot h_F\times h_F)+(\alpha\times\alpha)_\ast(A^2\cdot h_F\times h_F).$$
    Finally, since decomposable and generically defined cycles are closed under push-forward and pullback of products of maps defined in families, we obtain the result after rearranging and renaming terms. 
\end{proof}

\begin{proposition}\label{prop:linear-relation}
    Let $Y$ be a general double EPW quartic and $I$ as above. Then the class
    $$B\coloneq I+(\iota,\id)^\ast I \text{ in } \chow^2(Y\times Y), $$
    is decomposable and generically defined.
\end{proposition}
\begin{proof}
    First, notice that the right-hand side is generically defined, meaning that $B$ is too. 
    Now, the classes $c+c'$ in $\chow_1(X)$ with $\alpha(c)=y$ and $\alpha(c')=\iota(y)$ are hyperplane sections restricted to $D_{(L,M)}$ from $\tilde{Q}$.
    The $D_{(L,M)}$ are parametrised by $(\IP^2)^\vee\times(\IP^2)^\vee$ and as each~$D_{(L,M)}$ is a del Pezzo of degree $4$, we know that its hyperplane sections are parametrised by $\IP^3$, meaning that the classes $c+c'$ are parametrised by a $\IP^3$-bundle over $(\IP^2)^\vee\times(\IP^2)^\vee$.
    Thus, they are rationally equivalent, i.e. constant. Denote this class by $C=c+c'$.
    Then we have $p_\ast q^\ast c+p_\ast q^\ast c'=p_\ast q^\ast C$ for all $c,c'$ as above.
    Since~$I_F=(p\times p)_\ast(q\times q)^\ast\Delta_X$, we can rewrite this as 
    $$I|_{y\times Y}+I|_{\iota(y)\times Y}=\alpha_\ast (p_\ast q^\ast C\cdot h_F),$$
    where $y=\alpha(c)$.
    This means that the correspondence
    $$(I+(\iota,\id)^\ast I)_\ast\colon \chow_0(Y)\to \chow^2(Y)$$
    factors through the degree map. 
    Thus, by a decomposition of the diagonal argument, more precisely by \cite[Cor. 10.20]{voisinI}, we have 
    $$I+(\iota,\id)^\ast I=C'\times Y+\Gamma,$$
    where $\Gamma$ is supported on $Y\times D$ for some divisor $D\subset Y$. 
    Since $H^1(Y,\Ox_Y)=0$, it follows from \cite[Ex. III.12.6]{hartshorne} that $\chow^1(Y\times D)=\chow^1(Y)\oplus\chow^1(D)$ and $\Gamma$ is decomposable, completing the proof.
\end{proof}

We can now prove the main result of this section, which is a relation involving the projector onto the anti-invariant part of $\chow^\ast(Y)$, the incidence variety $I$, and decomposable generically defined cycles. 
This will allow us to prove some relations necessary for the Beauville--Voisin conjecture. 
\begin{corollary}\label{cor:proj-i}
    Let $Y$ be a double EPW quartic with associated Verra fourfold $X$ and let $I=(\alpha\times\alpha)_\ast(I_F\cdot h_F\times h_F)$, where $h_F$ is an ample divisor on $F(X)$. 
    Then there is a relation
    $$\frac{1}{2}\left(\Delta_Y-\Gamma_\iota\right)=I\cdot\Gamma+((\iota,\id)^\ast I)\cdot \Gamma'+\Gamma'' \text{ in }\chow^4(Y\times Y),$$
    where $\Gamma_\iota$ is the graph of the involution and $\Gamma,\Gamma',\Gamma''$ are decomposable and generically defined cycles.
\end{corollary}
\begin{proof}
    We square the relation from \Cref{prop:linear-relation} and obtain 
    $$(\iota, \id)^\ast I^2=I^2-2I\cdot B +B^2.$$
    Now, we can substitute the relation from \Cref{prop:I^2} to obtain
    $$(\iota,\id)^\ast(a\Gamma_\iota+ I\cdot \Gamma+\Gamma')=a\Gamma_\iota+I\cdot\Gamma+\Gamma' -2I\cdot B+B^2$$
    Since $\Gamma$ is decomposable and generically defined, it is $\iota$-invariant, meaning that 
    $$(\iota,\id)^\ast(I\cdot\Gamma)=(\iota,\id)^\ast(I)\cdot\Gamma.$$
    After rearranging and renaming, we obtain the statement. 
    This relies on the fact that the constant $a$ in \Cref{prop:I_F^2} and hence in \Cref{prop:I^2} is non-zero. For this, see the discussion in \Cref{rmk:non-zero}.
\end{proof}

\begin{corollary}\label{cor:chow^3_weak}
    Let $Y$ be a double EPW quartic. 
    Then we have
    $$\chow^3(Y)^-_{\mathrm{hom}}=\chow^1(Y)^+\cdot \chow^2(Y)^-_{\mathrm{hom}}.$$
\end{corollary}
\begin{proof}
    Notice that the classes that appear in \Cref{cor:proj-i} are all generically defined.
    Hence, they can be extended to all double EPW quartics with the relation still holding.
    Now, let $\gamma\in \chow^3(Y)_{\mathrm{hom}}^-$. 
    Applying the relation from \Cref{cor:proj-i}, we obtain
    $$\gamma=(I\cdot\Gamma)_\ast(\gamma)+(((\iota,\id)^\ast I)\cdot\Gamma')_\ast(\gamma),$$
    since the decomposable cycle $\Gamma''$ acts trivially on homologically trivial cycles. 
    We compute $(I\cdot\Gamma)_\ast(\gamma)$, as the other case is analogous. 
    We may write $$\Gamma=\pr_1^\ast a+\pr_2^\ast b+\sum_i\pr_1^\ast D_i\cdot\pr_2^\ast D_i'$$ for some $a,b\in\gdch^2(Y)$ and $D_i,D_i'\in\gdch^1(Y)=\chow^1(Y)^+$.
    Using the projection formula, we obtain 
    $$(I\cdot\Gamma)_\ast(\gamma)=\sum_i D_i'\cdot I_\ast(D_i\cdot\gamma),$$
    as $I_\ast(\gamma)\in\chow^1(Y)_{\mathrm{hom}}=0$ and $a\cdot\gamma=0$ for dimension reasons. 
    Finally, we have that~$I_\ast(D_i\cdot\gamma)\in\chow^2(Y)^-_{\mathrm{hom}}$.
    Applying the same argument to $(((\iota,\id)^\ast I)\cdot\Gamma')_\ast(\gamma)$ completes the proof.
\end{proof}

\section{\texorpdfstring{A constant cycle surface and $\chow_1(X)$}{A constant cycle surface and CH(X)}}\label{sec:constant-cycle-surface}

In this section we aim to prove that the fixed locus $Z\subset Y$ of the involution $\iota$ is a \emph{constant cycle surface}, meaning that all points lying on $Z$ have the same class in~$\chow_0(Y)$.
Along the way, we will study some products of codimension $2$ cycles and prove that the constant $a$ in \Cref{prop:I_F^2} is non-zero. 
Furthermore, we will study the correspondence $\chow_0(Y)\to\chow_1(X)$ induced by the universal conic and show that it induces an isomorphism $\chow_0(Y)^-\cong\chow_1(X)_{\mathrm{hom}}$ and use this to study $\chow^3(Y)^-_{\mathrm{hom}}$.
We will use the construction via conics in Verra fourfolds, see \Cref{sec:construction-verra}.

\subsection{The fixed locus as a constant cycle surface}\label{subsec:consta-cycle}

We begin by studying conics on the del Pezzo surfaces $D_{(L,M)}$.
For any conic $c\subset X$, consider the threefold $$D_c=\{c'\in~F~|~c'\cap c\neq \varnothing\}.$$ 
Considering $I_F$ as a correspondence, we have that $(I_F)_\ast([c])=D_c$ and $D_c$ only depends on the class $[c]\in\chow^3(X)$. 
Recall that $c$ is contained in $$D_{(L,M)}=\pi^{-1}(L\times M)\subset\IP(L\otimes M\oplus \IC),$$ which is a del Pezzo surface of degree 4. 
A general such del Pezzo is contained in 5 quadric threefolds $\tilde{Q}$ of rank $4$.
Any such $\tilde{Q}$ contains two pencils of planes, each of which intersects $D_{(L,M)}$ in a conic. 
One of these quadrics is the cone $C(L\times M)$ and the conics obtained from it are contracted under one of the two projections $X\to \IP^2\times\IP^2\to\IP^2$, meaning they are not $(1,1)$-conics. 
Conics obtained from pencils of planes on the other~$4$ quadrics are $(1,1)$-conics.
Recall also that the involution on $X$ induces an involution on~$D_{(L,M)}$, which interchanges any two pencils obtained from the same $\tilde{Q}$.

Thus, any $D_{(L,M)}$ contains $8$ pencils of $(1,1)$-conics. 
From the fact that any del Pezzo of degree $4$ is the blow-up of $\IP^2$ in $5$ points, we see that the pencils are $H-E_i$ and~$2H-\sum_{j\neq i}E_j$, $i=1,\ldots 5$.
This description depends on a chosen isomorphism and we cannot determine which two pencils consist of conics that are not $(1,1)$-conics, but we know that it must be $H-E_i$ and $2H-\sum_{j\neq i}E_j$ for the same $i$.
Denote the $8$ pencils by $\ell_i$, where $\ell_{i+8}=\ell_i$ and $\iota(\ell_i)=\ell_{i+4}$. 
From this description we can conclude that
$$c_i^2=0, ~c_i c_{i+4}=2, \text{ and }c_i c_j=1 \text{ for } j\neq i, i+4,$$
where $c_i$ lies on $\ell_i$. 
\begin{lemma}\label{lemma:conic-intersection}
Let $c$ be a general $(1,1)$-conic in a general Verra fourfold $X$. Let $D_{(L,M)}$ be the del Pezzo surface containing $c$ and assume $c\in\ell_i$. 
Then 
$$D_c^2=4\ell_{i+4}+\sum_{j\neq i,i+4}\ell_j \text{ in }\chow^4(F).$$
\end{lemma}
\begin{proof}
Consider a general $c_0\in\ell_i$. 
Then, $D_{c_0}$ and $D_c$ are rationally equivalent and~$D_c\cap~D_{c_0}$ consists of conics intersecting both $c$ and $c_0$. 
Let $c'\in D_c\cap D_{c_0}$ be one such conic with~$c'\subset D_{(L',M')}$. 
We begin by showing that $L=L'$ and $M=M'$. Suppose this were not the case and assume that $M\neq M'$. 
We know that two conics lying on the same pencil don't intersect, so the points of intersection $x=c'\cap c$ and $y=c'\cap c_0$ must be distinct. 
As $\pi(c')$ is a $(1,1)$-class, it intersects $L\times M\cap M'$ in a single point. 
Thus, the points $x$ and $y$ must have the same image in $\IP^2\times\IP^2$. Since $c'$ is a conic,~$\pi|_c$ is an isomorphism, as both $c'$ and $\pi(c')$ are of genus $0$, meaning that the above is a contradiction. 
Thus, $c'\in D_{(L,M)}$. 
Then, the calculations above show that 
$$\bigcup_{j\neq i}\ell_j=\supp D_c\cap D_{c_0}.$$

Now, the proof proceeds as in \cite[Lem. 3.9]{zhang24}. 
We provide it for the sake of completeness. 
As $D_c$ and $D_{c_0}$ intersect transversely, we only need to compute the multiplicity. 
Let $c'\in D_c\cap D_{c_0}$ be a general smooth conic and $D'_c=p^{-1}(D_c)$, and let~$D_{\tilde{c}}'=p^{-1}(D_{c_0})$. 
First, consider $c'\in D_c\cap D_{c_0}\cap \ell_{i+4}$. 
Let $x_1,x_2$ be the points of intersection of $c'$ and $c$ and $y_1,y_2$ those of $c'$ and $c_0$. 
Then, $p$ maps $(c',x_i)$ and $(c',y_i)$ to~$c'$. 
The tangent cones are
\begin{align*}
C_{c'}D_c=H^0(c',N_{c'/X}\otimes I_{x_1})\oplus H^0(c',N_{c'/D_{(L,M)}})\cup H^0(c',N_{c'/X}\otimes I_{x_2})\oplus H^0(c',N_{c'/D_{(L,M)}})\\
C_{c'}D_{c_0}=H^0(c',N_{c'/X}\otimes I_{y_1})\oplus H^0(c',N_{c'/D_{(L,M)}})\cup H^0(c',N_{c'/X}\otimes I_{y_2})\oplus H^0(c',N_{c'/D_{(L,M)}})
\end{align*}
Choose a general hypersurface $h_F\subset F$ whose class is ample that meets $\ell_{i+4}$ transversely at $c'$. 
Then, the multiplicity of $D_c|_{h_F}$ and $D_{c_0}|_{h_F}$ at $c'$ is $2$. 
Finally, notice that~$H^0(N_{c'/D_{(L,M)}})$ is contained in the kernel of the differential of $\alpha$, i.e. is the same tangent direction as $\ell_{i+4}$. 
As $h_F$ meets $\ell_{i+4}$ transversely, we have
$$C_{c'}D_c|_{h_F}\cap C_{c'}D_{\tilde{c}}|_{h_F}=0$$ 
and using \cite[Prop. 1.29]{3264}, we obtain that the multiplicity of $D_c$ and $D_{c_0}$ at $\ell_{i+4}$ is~$4$.

Finally, consider $c'\in D_c\cap D_{c_0}\cap \ell_j$ for some $j\neq i+4$. 
The conic $c'$ intersects both $c$ and $c_0$ in a single point $x$ and $y$, respectively. 
Then, from the proof of \Cref{lemma:normal-bundle}, we see that 
$$T_{(c',x)}D'_c=H^0(c', N_{c'/X}\otimes I_x)\oplus H^0(c', N_{c'/D_{(L,M)},x})$$
and hence that $D'_c$ is smooth at $c'$.
Thus, 
$$T_{c'}D_c\cap T_{c'} D_{c_0}=H^0(c',N_{c'/D_{(L,M)}})=T_c\ell_j$$
and $D_c$ and $D_{c_0}$ intersect transversely at $c'$, completing the proof.
\end{proof}
Let $c_i\in\ell_i$. 
Applying the above to all conics contained in the same $D_{(L,M)}$, we obtain the following.
\begin{equation}
5\left( D_{c_i}^2+D_{c_{i+4}}^2\right)-\sum_{j=1}^8 D_{c_j}^2=10\left(\ell_i+\ell_{i+4} \right)
\tag{$\star$}\label{eq:star}
\end{equation}
\begin{lemma}\label{lemma:non-zero}
The constant $a$ in \Cref{prop:I_F^2} is non-zero. 
Furthermore, if $[c]=[c']$ in $\chow_1(X)$, then $[\alpha(c)]=[\alpha(c')]$ in $\chow_0(Y)$.
\end{lemma}
\begin{proof}
On the contrary, assume that $a=0$. 
Then,
$$I_F^2=I_F\cdot\Gamma+\Gamma'$$
for some decomposable and generically defined $\Gamma,\Gamma'$. 
The map $\chow_0(F)\to\chow^\ast(F)$ induced by a decomposable correspondence depends only on the degree of the zero-cycle and $I_\ast$ factors through $\chow_1(X)$.
Hence,
$$(I^2_F)_\ast(\xi)=(I^2_F)_\ast(\zeta)$$
if $q_\ast p^\ast  \xi = q_\ast p^\ast \zeta$ in $\chow_1(X)$. Using \cite[Lem. 17.3]{shenvial}, we obtain $(I_F^2)_\ast(c)=D_c^2$. 
The argument in the proof of \Cref{prop:linear-relation} shows that $c+c'$ is constant in $\chow_1(X)$ for any $c$ and $c'$ with $\alpha(c')=\iota(\alpha(c)))$.
Hence, $D_c^2+D_{c'}^2$ doesn't depend on $c$, where $c'$ is some conic satisfying $\alpha(c')=\iota(\alpha(c))$.

Now, choose a very ample divisor $h_F$ on $F$ and let $m$ be the intersection number of~$h_F$ with a general fibre of $\alpha$. 
Then it follows from \Cref{eq:star} that
$$\alpha_\ast((5(D_c^2+D_{\iota(c)}^2)-\sum_{i=1}^8 D_{c_j}^2)\cdot h_F)=\alpha_\ast(10(\ell+\ell')\cdot h_F)=10m([\alpha (c)]+[\iota(\alpha(c))]),$$
implying that $[\alpha(c)]+[\iota(\alpha(c))]$ is constant. 
This implies that $\Gamma_\iota$, the graph of the involution $\iota$, is a constant-cycle subvariety in $Y\times Y$. Let $\sigma$ be the symplectic form on~$Y$. 
By Mumford's theorem, see \cite[Cor. 10.18]{voisinI}, $(\pr_1^\ast\sigma^2+\pr_2^\ast\sigma^2)|_{\Gamma_\iota}=0$. 
Since $\iota$ is anti-symplectic, $2\pr_1^\ast\sigma^2|_{\Gamma_\iota}=0$, which implies $\sigma^2=0$, as $\pr_1\colon \Gamma_\iota\to Y$ is an isomorphism. 
This is a contradiction, and $a$ is non-zero.

For the second statement, let $[c]=[c']$. 
Then, we have 
\begin{align*}
0&= \alpha_\ast(h_F\cdot ((I_F^2)_\ast[c]-(I_F^2)_\ast[c']))\\
&= \alpha_\ast(h_F\cdot (aW+I_F\cdot\Gamma+\Gamma')_\ast([c]-[c']))\\
&= \alpha_\ast(h_F\cdot aW_\ast([c]-[c']))\\
&= am([\iota(\alpha(c))]-[\iota(\alpha(c'))]).
\end{align*}
Thus, $[c]=[c']$.
\end{proof}
We can now prove that the fixed locus $Z$ of the involution $\iota$ is a constant-cycle surface in $Y$. 
Consider the correspondence
$$\Psi\coloneq P_\ast\circ\alpha_\ast^{-1}=\frac{1}{m}P_\ast\circ (\pr_2^\ast h_F)_\ast\circ \alpha^\ast\colon \chow_0(Y)\to \chow_1(X),$$
where $h_F$ is an ample divisor on $F(X)$, $m$ is the intersection number of $h_F$ with the general fibre of $\alpha$, and $P$ is the universal conic $P\subset F(X)\times X$. 

\begin{theorem}\label{thm:constant-cycle}
    Let $Y$ be a general double EPW quartic with associated Verra fourfold~$X$.
    Then, $\Psi(y+\iota(y))$ does not depend on $y\in Y$. 
    In particular, $\Psi(2z)$ is constant in $\chow_1(X)$ for any $z\in Z$. 
    Furthermore, $Z$ is a constant cycle surface and any $z\in Z$ has class $o_Y$, the Beauville--Voisin class.
\end{theorem}
\begin{proof}
    It suffices to prove this for a general point $y\in Y$, so we may assume that $y\notin Z$. 
    Recall from the proof of~\Cref{prop:linear-relation} that the class $c+c'\in\chow^3(X)$ with $\alpha(c)=y$ and $\alpha(c')=\iota(y)$ is constant, meaning that $\Psi(y+\iota(y))$ is constant.
    This holds for all points, and in particular for $c\in \alpha^{-1}(Z)$, meaning that the class $[2c]\in\chow^3(X)$ doesn't depend on $c$, and using \Cref{lemma:non-zero}, the class $[\alpha(c)]=[z]$ is constant in $Y$ for any $z\in Z$.
    Since $Z$ is generically defined, the class of any point on $Z$ is generically defined, meaning that it agrees with $o_Y$ by \Cref{thm:franchetta}.
\end{proof}

In view of \cite[Conj. 0.8]{voisin16}, we would expect points lying on any constant cycle surface to have the same class. 

\begin{proposition}\label{prop:nef}
    Let $Y$ be a double EPW quartic.
    Any point lying on a constant cycle surface in $Y$ has class $o_Y$.
\end{proposition}
\begin{proof}
    Let $Z\subset Y$ be the fixed locus of the involution. 
    Then for the general double EPW quartic, any point on $Z$ has class $o_Y$ by \Cref{thm:constant-cycle}.
    Since $Z$ is generically defined, this extends to any $Y$.
    Finally, as $Y$ admits a Lagrangian fibration, \cite[Prop. 3.1]{twistedSYZ} shows that points on constant cycle Lagrangian subvarieties have the same class.
\end{proof}

\comment{
\begin{proposition}
    Let $Y$ be a very general double EPW quartic and let $Z'\subset Y$ be another constant cycle surface. 
    Then, the class of any point $z'\in Z'$ is $o_Y$.
\end{proposition}
\begin{proof}
    It suffices to prove that any two surfaces in $Y$ intersect, which we show by proving that any two surfaces have positive intersection number. 
    Let $N_2(Y)=\chow_2(Y)\otimes\IR/\sim_{\mathrm{num}}$ be the group of $2$-cycles modulo numerical equivalence. 
    Since $Y$ is very general, this is a $4$-dimensional vector space spanned by $h_1^2, h_2^2, h_1h_2,$ and $Z$. 
    Given that $h_1^2$ and $h_2^2$ are represented by Lagrangian subvarieties obtained from a Lagrangian fibration, we may apply \cite[Prop. 3.3 and Prop. 3.4]{ottem22}, which show that $c_2(Y), h_1^2,$ and $h_2^2$ are not contained in the interior of the effective cone $\overline{\mathrm{Eff}}_2(Y)$. 
    Using \Cref{lemma:class-of-Z} we can compute the following intersection numbers.
    $$Z^2=1404, h_1^2Z=48, h_2^2Z=48, c_2(Y)Z=72$$
    From this we obtain that the cone generated by $h_1^2, h_2^2, c_2(Y),$ and $Z$ is contained in its dual cone,
    $$\mathrm{cone}(h_1^2, h_2^2, c_2(Y), Z)\subset \mathrm{cone}(h_1^2, h_2^2, c_2(Y), Z)^\vee\subset \mathrm{Nef}_2(Y).$$
    As all the intersection numbers are strictly positive, the first cone is contained entirely in the interior of the second cone, meaning that any surface on $Y$ is strictly nef, i.e. has positive intersection with any other surface.
\end{proof}
}

\subsection{\texorpdfstring{Comparing $\chow_0(Y)$ and $\chow_1(X)$}{Comparing CH(Y) and CH(X)}}\label{subsec:constant2}

We now study the correspondence $\Psi$ in more detail. 
First, notice that we have 
$$\Psi^t\circ\Psi=((\alpha\times\alpha)_\ast(I_F\cdot h_F\times h_F))_\ast=I_\ast.$$
We will use the calculation of $D_c^2$ in \Cref{lemma:conic-intersection} to show that the correspondence $\Psi$ induces an isomorphism $\chow_0(Y)^-\cong\chow_1(X)_{\mathrm{hom}}$.

Let $y\in Y$ and define the following class in $\chow^2(Y)$:
$$S_y\coloneq I_\ast(y)=\frac{1}{m}\alpha_\ast(D_c\cdot h_F),$$
where $c$ is a conic with $\alpha(c)=y$. 

\begin{lemma}\label{lemma:D-S}
    The difference
    $$A\coloneq D_c-\alpha^\ast S_{\alpha(c)} \text{ in } \chow^2(F(X))$$
    is generically defined and we have
    $$S_y\cdot S_{y'}=\alpha_\ast(D_c \cdot D_{c'}\cdot h_F)-\alpha_\ast(A^2 \cdot h_F).$$
    Furthermore, $S_y+S_{\iota(y)}$ is generically defined and independent of $y\in Y$.
\end{lemma}
\begin{proof}
    That the difference $A$ is generically defined follows directly from \Cref{pullback-I} and we can use the projection formula to compute $S_y\cdot S_{y'}$. 
    That $S_y+S_{\iota(y)}$ is independent of~$y\in Y$ follows from \Cref{thm:constant-cycle}, as $S_y+S_{\iota(y)}=(\Psi^t\circ\Psi)(y+\iota(y))$. 
    Hence,~$S_y+~S_{\iota(y)}=~I_\ast(2o_Y)$ and it is generically defined. 
\end{proof}

\begin{theorem}\label{thm:chow-iso}
    Let $Y$ be a general double EPW quartic with associated Verra fourfold~$X$.
    The correspondence $\Psi$ induced by the universal conic is zero on $\chow_0(Y)^+_{\mathrm{hom}}$ and induces an isomorphism 
    $$\Psi\colon \chow_0(Y)^-\xrightarrow{\cong} \chow_1(X)_{\mathrm{hom}}.$$
    Its inverse is given up to a scalar by $\Psi^t$ composed with multiplication by $I_\ast(o_Y)$.
\end{theorem}
\begin{proof}
    We know that $\chow_0(Y)^+_{\mathrm{hom}}$ is generated by classes of the form $y+\iota(y)-2o_Y$ for~$y\in Y$, hence $\Psi|_{\chow_0(Y)^+_{\mathrm{hom}}}=0$ by \Cref{thm:constant-cycle}. 
    Notice that $\Psi$ is surjective by \cite[Cor. 4.4]{laterveer-verra}. 
    Hence, we only need to check injectivity on $\chow_0(Y)^-_{\mathrm{hom}}$. 
    To this end, consider points $y_i, y_i'$ such that
    $$\Psi\left(\sum y_i-\iota(y_i)\right)=\Psi\left(\sum y_i'-\iota(y_i')\right) $$
    Applying $P^t_\ast$ we obtain
    $$\sum D_{c_i}-D_{\iota(c_i)}=\sum D_{c_i'}-D_{\iota(c_i')},$$
    where $c_i$, $c_i'$, $\iota(c_i)$, and $\iota(c_i')$ are conics that lift $y_i, y_i', \iota(y_i)$, and $\iota(y_i')$, respectively. 
    From \Cref{thm:constant-cycle} we have that $D_{c}+D_{\iota(c)}$ is constant for any conic $c$.
    Using \Cref{lemma:conic-intersection,lemma:D-S} after multiplying the last equation with $A$ and $h_F$ and pushing forward, we obtain that
    $$\alpha_\ast\left(h_F\cdot (D_c+D_{\iota(c)})\cdot \sum D_{c_i}-D_{\iota(c_i')}\right)=S_y^2-S_{\iota(y)}^2=4m\sum_i (y_i-\iota(y_i)),$$
    where $m$ is the intersection number of $h_F$ with the general fibre of $\alpha$.
    From this we obtain that $\sum_i y_i-\iota(y_i)=\sum_i y_i'-\iota(y_i')$ and $\Psi$ is injective. 
    Combining the calculation above with the fact that $\Psi^t\circ\Psi=I_\ast$ and $S_y+S_{\iota(y)}=2I_\ast(o_Y)$ by \Cref{thm:constant-cycle} shows that the inverse is as stated.
\end{proof}

\begin{corollary}\label{cor:chow-iso}
    Let $Y$ be a double EPW quartic. 
    There are correspondences that are inverses up to a scalar,
    $$(\pr_1^\ast I_\ast(o_Y))_\ast\colon\chow^2(Y)^-_{\mathrm{hom}} \mathrel{\substack{\longrightarrow \\[-1pt] \longleftarrow}} 
    \chow_0(Y)^-\colon I_\ast, $$
    given by multiplication by $I_\ast(o_Y)$ and $I_\ast$. 
    In particular, $\chow^2(Y)^-_{\mathrm{hom}}\cong\chow^2(Y)^-_{\mathrm{alg}}$, i.e. rational and algebraic equivalence agree on $\chow^2(Y)^-$.
\end{corollary}
\begin{proof}
    First, let $Y$ be a general double EPW quartic.
    Since $\Psi^t\circ\Psi=I_\ast$, we obtain from \Cref{thm:chow-iso} that $I_\ast$ is an isomorphism with inverse given up to a scalar by multiplication by~$I_\ast(o_Y)$.
    The classes $I$ and $I_\ast(o_Y)$ are generically defined, meaning that they extend to any double EPW quartic and remain inverse correspondences.
\end{proof}

\subsection{\texorpdfstring{The motives of $X$ and $Y$}{The motives of X and Y}}\label{subsec:motives}

The aim of this subsection is to compare the motives of $X$ and $Y$, in particular their transcendental motives. 
We begin by studying the action of the universal conic on cohomology, before studying $\chow^3(Y)^-_{\mathrm{hom}}$ and combining these results to show that there is an isomorphism of transcendental motives 
$$\mathfrak{h}(Y)^-\cong t(X)(-1)\oplus t(X)^{\oplus 2}\oplus t(X)(1)\oplus \bigoplus_j \mathds{1}(d_j)^{\oplus r_j}.$$

\begin{remark}\label{rmk:verra-fix}
    The proof of \cite[Prop. 2.8]{laterveer-verra}, on which many of the results in \emph{loc. cit.} rely, unfortunately contains a mistake. 
    The universal conic is claimed to be a projective bundle and the projective bundle formula is used prominently. 
    This is not the case, as Verra fourfolds contain singular conics. 
    We provide a fix to this below and note that one can alternatively also proceed as in  \cite[Sec. 4.2]{ilievmanivel} and show that the holomorphic two-form on $F(X)$, that descends to $Y$, is induced from $X$. 
    This can be used to compute the symplectic form on $Y$ explicitly and provide a geometric argument for the involution~$\iota$ being anti-symplectic. 
\end{remark}

\begin{proposition}[{\cite[Prop. 2.8]{laterveer-verra}}]\label{prop:verra-fix}
    Let $X$ be a general Verra fourfold and $Y$ the associated double EPW quartic. 
    Let $p\colon P\to F(X)$ be the universal conic and $q\colon P\to X$ the other projection. 
    Then the composition induces an isomorphism 
    $$p_\ast q^\ast\colon H^4_{\mathrm{pr}}(X,\IQ)\xrightarrow{\cong}\im(H^2_{\mathrm{pr}}(Y,\IQ)\xrightarrow{\alpha^\ast} H^2(F(X),\IQ)).$$
    Furthermore, the induced isomorphism $\psi\colon H^4_{\mathrm{pr}}(X)\to H^2_{\mathrm{pr}}(Y)$ satisfies 
    $$\lambda(x,y)=q(\psi(x),\psi(y))\text{ for }x,y\in H^4_{\mathrm{pr}}(X) \text{ and some }\lambda\neq 0,$$
    where $(x,y)$ denotes the intersection form and $q$ the Beauville--Bogomolov form.
\end{proposition}
\begin{proof}
    It is sufficient to prove the statement for very general $X$.
    We claim that the map being non-zero is sufficient in this case.
    Indeed, notice that $H^4_{\mathrm{pr}}(X)=H^4_{\mathrm{tr}}(X)$ for very general $X$ and that this is a simple Hodge structure. 
    Thus, since $H^{3,1}(X)$ maps to $H^{2,0}(F(X))=\alpha^\ast H^{2,0}(Y)$, the image of the map is contained in $\alpha^\ast H^2(Y)$ if the map is non-zero.
    Furthermore, there is an isomorphism $H^4_{\mathrm{pr}}(X)\cong H^2_{\mathrm{pr}}(Y)$ via a different construction that satisfies the equation above \cite[Thm. 2.4]{laterveer-verra}. 
    Composing its inverse with $p_\ast q^\ast$ grants an endomorphism of $H^4_{\mathrm{tr}}(X)$, which is a multiple of the identity by the simplicity of $H^4_{\mathrm{tr}}(X)$.
    Hence, if $p_\ast q^\ast$ non-zero, this composition is a non-zero multiple and an $p_\ast q^\ast$ is an isomorphism that satisfies the equation above.

    We prove that $p_\ast q^\ast$ is non-zero. 
    Recall that $H^k_{\mathrm{tr}}(X)= H^k(X)/ N^1 H^k(X)$, where 
    $$ N^1H^k(X)\coloneq \sum_{D\subset X} \ker(H^k(X)\to H^k(X\setminus D))$$
    and the sum runs over all divisors $D\subset X$.
    As the fibres of $p$ have trivial $\chow_0$, it follows from \cite[Prop. 3.1]{vial-fibrations} that $p_\ast$ induces an isomorphism $\chow_0(P)\cong\chow_0(F(X))$.
    A decomposition of the diagonal argument shows that $p_\ast$ induces an isomorphism 
    $$p_\ast \colon H^4_{\mathrm{tr}}(P)\xrightarrow{\cong} H^2_{\mathrm{tr}}(F(X))$$
    with inverse given by $p^\ast$ composed with multiplication by the class of an ample line bundle. 
    Thus, for any $x\in H^4_{\mathrm{tr}}(X)$, we have 
    $$q^\ast x=p^\ast x_0\cdot h_P\text{ for an ample }h_P\text{ and }x_0\in H_{\mathrm{tr}}^2(F(X)).$$
    Let $F'$ be the intersection of two general hyperplane sections in $F(X)$ and $P'$ the restriction of $P$ to $F'$. 
    Denote the restrictions of $p$ and $q$ to $P'$ by $p'$ and $q'$ and the inclusion by $j\colon F'\to F(X)$.
    Then 
    $$j_\ast(p')_\ast(q')^\ast x=p_\ast q^\ast x \cdot h_F^2.$$
    Since multiplication by $h_F^2$ is injective by the Hard Lefschetz Theorem, it suffices to prove that $(p')_\ast(q')^\ast$ is non-zero.
    Notice that $q'$ is generically finite, as the fibres of $q$ are two-dimensional and a general intersection of two hyperplane sections will intersect the general fibre in finitely many points. 
    Let $d$ be the degree of $q'$.
    Squaring the equation~$(q')^\ast x=(p')^\ast x_0\cdot h_{P'}$ and pushing forward to $F'$, we obtain 
    $$ d(x,x)=(p')_\ast (q')(x^2)=x_0^2\cdot (p')_\ast h_{P'}^2.$$
    It also follows from the equation above that $x_0=p_\ast q^\ast$.
    In particular, the morphism is non-zero, as the intersection pairing is non-degenerate on $H^4_{\mathrm{tr}}(X)$, completing the proof.
\end{proof}

\Cref{cor:chow-iso} shows that the map $\chow^2(Y)^-_{\mathrm{hom}}\to\chow_0(Y)^-$ induced by multiplication by $I_\ast(o_Y)$ is an isomorphism.
The Bloch--Beilinson conjecture implies that multiplication with $h^2$, the square of the polarisation, should also define an isomorphism as above, which we prove below.
This is analogous to the case of Fano varieties of lines on cubic fourfolds, cf. \cite[Props. 22.2 and~22.3]{shenvial}.
Let $h_1$ and $h_2$ be the divisors pulled back from $\IP^2$ along the Lagrangian fibrations $Y\to\IP^2$. 
Then let $h=h_1+h_2$ be the polarisation and $g=h_1-h_2$.

\begin{proposition}\label{prop:h^2}
    Let $Y$ be a double EPW quartic.
    Then, multiplication by $h_1^2$ and~$h_2^2$ is~$0$ on $\chow^2(Y)^-_{\mathrm{hom}}$.
    Furthermore, the map $\chow^2(Y)^-_{\mathrm{hom}}\to~\chow_0(Y)^-$ induced by multiplication by $h^2$ is an isomorphism, which agrees up to a scalar with the isomorphism induced by the multiplication with $I_\ast(o_Y)$ from \Cref{thm:chow-iso}.
\end{proposition}
\begin{proof}
    Since all the classes involved are generically defined, it suffices to prove the statement for the general double EPW quartic, so we may assume that $Y$ admits an associated Verra fourfold $X$.

    Consider the composition
    $$
    \chow_1(X)_{\mathrm{hom}}\xrightarrow{\Psi}\chow^2(Y)^-_{\mathrm{hom}}\xrightarrow{\cdot h_i^2}\chow_0(Y)^-\xrightarrow{\Psi^t}\chow_1(X)_{\mathrm{hom}}
    $$
    Since $\Psi$ and $\Psi^t$ are isomorphisms by \Cref{thm:chow-iso}, it suffices to prove that the composition is zero.
    Using the Franchetta property for $X\times X$ from \Cref{thm:franchetta}, this follows from knowing that the composition is $0$ in cohomology. 

    Define the primitive motive $\mathfrak
    {h}_{\mathrm{pr}}^4(X)$ of $X$ as the motive cut out by the idempotent correspondence 
    $$\pi^4_{\mathrm{pr}}=\Delta_X-X\times x-x\times X-g_1\times g_1 g_2^2- g_2\times g_2 g_1^2-\frac{g_1^2\times g_1^2}{\deg g_1^2 g_2^2}-\frac{g_2^2\times g_2^2}{\deg g_1^2 g_2^2},$$
    where the $g_i$ are the pullbacks of the hyperplane classes along the projections $X\to \IP^2$ and $x$ is a point.
    It satisfies $H^k(\mathfrak{h}_{\mathrm{pr}}^4(X))=H^k_{\mathrm{pr}}(X)$, which is $0$ for $k\neq 4$.
    Hence the composition becomes the following in cohomology.
    $$
    H^4_{\mathrm{pr}}(X)\xrightarrow{P_\ast^t} H^2_{\mathrm{pr}}(Y)\xrightarrow{\cdot h_i^2} H^6_{\mathrm{pr}}(Y)\xrightarrow{P_\ast}H^4_{\mathrm{pr}}(X)
    $$
    Since $q(h_i,h_i)=0$ and $q(h_i, D)=0$ for all $D\in H^2_{\mathrm{pr}}(Y)$, multiplication by $h_i^2$ vanishes on $H^2_{\mathrm{pr}}(Y)$, meaning that the composition also vanishes in cohomology.

    Finally, since $I_\ast(o_Y)$ is generically defined, it is a linear combination of $h_1^2, h_1h_2, h_2^2,$ and $c_2(Y)$. 
    By \Cref{lemma:class-of-Z}, the proof of which is independent of this result, the class of~$Z$ is linearly independent of $h_1^2, h_1h_2, h_2^2$. 
    For any $\alpha\in\chow^2(Y)^-_{\mathrm{hom}}$, $Z\cdot\alpha$ is supported on~$Z$ and hence is a multiple of $o_Y$ by \Cref{thm:constant-cycle}.
    Thus, it is equal to $0$, as it is also anti-invariant. 
    Hence, multiplication with $I_\ast(o_Y)$ agrees up to a scalar with multiplication with $h_1h_2$, which in turn agrees up to a scalar with multiplication by $h^2$, completing the proof. 
\end{proof}

\begin{corollary}\label{cor:chow^3_strong}
    Let $Y$ be a double EPW quartic.
    Then multiplication by $g^2$ induces an isomorphism $\chow^2(Y)^-_{\mathrm{hom}}\cong \chow_0(Y)^-$ with inverse $I_\ast$ and
    $$\chow^3(Y)^-_{\mathrm{hom}}=h\cdot\chow^2(Y)^-_{\mathrm{hom}}\oplus g\cdot\chow^2(Y)^-_{\mathrm{hom}}.$$
\end{corollary}
\begin{proof}
    Since $g^2=h_1^2-2h_1h_2+h_2^2$, it follows from \Cref{prop:h^2} that multiplication by $g^2$ agrees with multiplication by $-h^2$ and is an isomorphism.
    We know from \Cref{cor:chow^3_weak} that $$\chow^3(Y)^-_{\mathrm{hom}}=\chow^1(Y)^+\cdot\chow^2(Y)^-_{\mathrm{hom}}=h\cdot\chow^2(Y)^-_{\mathrm{hom}}+g\cdot\chow^2(Y)^-_{\mathrm{hom}}.$$
    To show that this is indeed a direct sum, consider $\alpha,\beta\in\chow^2(Y)^-_{\mathrm{hom}}$ such that $\alpha\cdot h=\beta \cdot g$.
    Multiplying with $h$ shows that $\alpha\cdot h^2=\beta\cdot gh=\beta\cdot(h_1^2-h_2^2)=0$ by \Cref{prop:h^2}. 
    Since multiplication with $h^2$ is an isomorphism, $\alpha=0$.    
\end{proof}

\begin{corollary}
    Let $Y$ be a general double EPW quartic with associated Verra fourfold~$X$ and $Y\to D_1^{\overline{A}}$ be the double cover. 
    There is an isomorphism of rational Chow motives 
    $$\mathfrak{h}(Y)\cong \mathfrak{h}(D_1^{\overline{A}})\oplus t(X)(-1)\oplus t(X)^{\oplus 2}\oplus t(X)(1)\oplus \bigoplus_j \mathds{1}(d_j)^{\oplus r_j},$$
    where $t(X)$ denotes the transcendental motive of $X$.
\end{corollary}
\begin{proof}
    It is sufficient to prove that 
    $$\mathfrak{h}(Y)^-\cong t(X)(-1)\oplus t(X)^{\oplus 2}\oplus t(X)(1)\oplus \bigoplus_j \mathds{1}(d_j)^{\oplus r_j}.$$
    Since we are working over $\IC$, which is a universal domain, \cite[Lem. 1.1]{huybrechts} reduces this to proving that there is an isomorphism of Chow groups
    $$\chow^\ast(Y)^-\cong \chow^\ast(t(X)(-1)\oplus t(X)^{\oplus 2}\oplus t(X)(1)\oplus\bigoplus_j\mathds{1}(d_j)^{\oplus r_j})$$
    induced by correspondences. 
    These are given by the isomorphisms in \Cref{prop:verra-fix,cor:chow-iso,cor:chow^3_strong} composed with the isomorphism $\chow_0(Y)^-\cong\chow_1(X)_{\mathrm{hom}}$ from \Cref{thm:chow-iso}.
\end{proof}

\subsection{Voisin's filtration on zero-cycles}
Voisin studied Beauville's conjecture on the splitting of the conjectural Bloch--Beilinson filtration in \cite{voisin16} and defined a filtration opposite to it. 
\begin{definition}[\cite{voisin16}]
    Let $T$ be a hyperkähler variety of dimension $n$ and let $O_x$ be the set of points rationally equivalent to $x\in T$.
    This is a countable union of subvarieties of $T$ and has a well-defined dimension as the supremum over the dimensions of these subvarieties.
    Define \emph{Voisin's filtration} on $\chow_0(T)$ as
    $$
    S_i^{\text{V}}\chow_0(T)\coloneq \left\langle [x]~|~x\in T \text{ and }\dim O_x\geq n-i\right\rangle.
    $$
\end{definition}

Adapting arguments from \cite{bolognesi-laterveer} to the setting of double EPW quartics, we prove that Voisin's filtration is of the form predicted by the Bloch--Beilinson conjecture.

\begin{proposition}\label{prop:filtration}
    Let $Y$ be a double EPW quartic. 
    Then $S_1^{\text{V}}\cap\chow_0(Y)_{\mathrm{hom}}=\chow_0(Y)^-$ and Voisin's filtration agrees with the filtration obtained from the involution $\iota$,
    $$
    \IQ o_Y\subset \IQ o_Y\oplus\chow_0(Y)^-\subset \chow_0(Y)
    $$
\end{proposition}
\begin{proof}
    Any point with $y\in Y$ with $\dim O_y=2$ lies on a Lagrangian constant cycle subvariety, meaning that its class is $o_Y$ by \Cref{prop:nef}.
    Hence, $S_0^{\text{V}}\chow_0(Y)=~\IQ o_Y$.
    It follows from \cite[Thm. 1.8]{charles-mongardi-piacenza} that $S_0^{\text{V}}\chow_0(Y)=D\cdot\chow^3(Y)$ for any $D\in\chow^1(Y)$.
    We obtain from \Cref{prop:h^2} that 
    $$
    \chow_0(Y)^-=h^2\cdot\chow^2(Y)^-_{\mathrm{hom}}\subset h\cdot\chow^3(Y)=S_1^{\text{V}}\chow_0(Y).
    $$
    Now, consider the double cover $f\colon Y\to D_1^{\overline{A}}$ and the inclusion $j\colon D_1^{\overline{A}}\to C(\IP^2\times\IP^2)$.
    Let~$\alpha\in \chow^3(Y)^+$. 
    Then 
    $$
    h\cdot\alpha=\frac{1}{2}f^\ast f_\ast h\cdot\alpha=\frac{1}{2}f^\ast(h\cdot f_\ast\alpha)=\frac{1}{8}f^\ast j^\ast j_\ast f_\ast\alpha=\lambda o_Y,
    $$
    since $D_1^{\overline{A}}$ is a quartic section in the cone and $j^\ast j_\ast f_\ast\alpha\in j^\ast\chow^4(C(\IP^2\times\IP^2))=\IQ o_Y$.
    Thus, 
    $$
    S_1^{\text{V}}\chow_0(Y)\cap \chow_0(Y)_{\mathrm{hom}}=h\cdot(\chow^3(Y)^+\oplus\chow^3(Y)^-)\cap\chow_0(Y)_{\mathrm{hom}}\subset\chow_0(Y)^-
    $$
\end{proof}

\section{A multiplicativity result implied by the Bloch--Beilinson conjecture}\label{sec:multiplicativity}

The aim of this section is to prove that the multiplication 
$$\chow^2(Y)^-_{\mathrm{hom}}\otimes\chow^2(Y)^-_{\mathrm{hom}}\to\chow_0(Y)^+_{\mathrm{hom}}$$
is surjective. 
Since $H^0(Y,\Omega_Y^2)\otimes H^0(Y,\Omega_Y^2)\to H^0(Y,\Omega_Y^4)$ is surjective, the generalised Hodge conjecture and the Bloch--Beilinson conjecture predict that the multiplication map $F^2\chow^2(Y)\otimes F^2\chow^2(Y)\to F^4\chow_0(Y)$ is surjective.
Then, the expected description of the Bloch--Beilinson filtration on $Y$ would imply the statement above. 
The analogous statement was shown in the case of Fano varieties of lines on cubic fourfolds for what the filtration is expected to be \cite[Thm. 20.2]{shenvial}.

To prove this result, we will need to compute more intersection products of varieties of the type $D_c$, as in \Cref{subsec:consta-cycle}, meaning that we will be using similar methods.
We need to understand more products of the type $S_y\cdot S_{y'}$, where
$$S_y\coloneq I_\ast(y)=\frac{1}{m}\alpha_\ast(D_c\cdot h_F).$$
Recall that $h_F$ is a choice of ample divisor on $F(X)$, $m$ is the intersection number of $h_F$ with the general fibre of $\alpha$, and
$$D_c=\left\{c'\in F(X)~|~c'\cap c\neq\varnothing\right\}.$$
We begin by computing the product of $D_c\cdot D_{c'}$, where $c'$ is any conic with~$\alpha(c')=~\iota(\alpha(c))$, which is similar to \cite[Prop. 20.7]{shenvial}.
Let $\pi\colon X\to \IP^2\times\IP^2$ be the double cover and recall that any $c$ is contained in a del Pezzo surface $D_{(L,M)}=\pi^{-1}(L\times M)$.
Any such del Pezzo surface contains $8$ pencils of~$(1,1)$-conics, which we denote $\ell_i$. 
For more details, see the discussion at the beginning of \Cref{subsec:consta-cycle}.
\begin{proposition}\label{prop:Dc.Diota}
    Let $X$ be a general Verra fourfold and $c$ a general $(1,1)$-conic contained in $D_{(L,M)}\subset X$ and assume $c\in \ell_i$.
    Let $c'$ be a general conic lying on $\ell_{i+4}$, i.e~$\alpha(c')=~\iota(\alpha(c))$.
    Then
    $$D_c\cdot D_{c'}=\Gamma+\sum_{j\neq i,i+4}\ell_j \text{ in }\chow^4(F),$$
    where $\Gamma$ is generically defined. 
\end{proposition}
\begin{proof}
    We begin by computing the set-theoretic intersection $D_c\cap D_{c'}$. First, notice that~$c$ and $c'$ intersect in two points, which we denote $x_1$ and $x_2$. 
    Let $E_x$ be the surface of conics passing through a point $x\in X$,
    $$E_x\coloneq p(q^{-1}(x))=\left\{c\in F(X)~|~x\in c\right\}.$$
    Clearly, $E_{x_1}\cup E_{x_2}\subset D_c\cap D_{c'}.$ 
    We also know that any conic  $\tilde{c}\in\ell_j$ for $j\neq i,i+4$ intersects $c$ and $c'$ in one point, meaning that
    $$T\coloneq E_{x_1}\cup E_{x_2}\cup\bigcup_{j\neq i, i+4}\ell_j\subset D_c\cap D_{c'}.$$
    We claim that this is an equality. 
    Indeed, any conic $\tilde{c}$ contained in $D_{(L,M)}$ that intersects both $c$ and $c'$ must lie on one of the $\ell_j$ with $j\neq i, i+4$, so we now consider a conic $\tilde{c}$ contained in $D_{(L',M')}$ with $L'\neq L$. 
    Then, since $\pi(\tilde{c})$ is a $(1,1)$-class, $\pi(\tilde{c})\cap (L\cap L'\times M')$ consists of a single point. 
    Thus, $\pi(\tilde{c}\cap c)=\pi(\tilde{c}\cap \iota(c))$, which means that $\tilde{c}\cap c= \tilde{c}\cap c'$, as $\pi|_c$ is an isomorphism. 
    Thus, $\tilde{c}\in E_{x_1}\cup E_{x_2}$, proving the claim. 

    Now we use the excess intersection formula to compute the cycle-theoretic intersection, see \cite[Thm. 9.2]{fulton} or \cite[Sec. 13.3]{3264}. 
    First, we know that the $\ell_j$ are all disjoint from each other, but they intersect each $E_{x_j}$ in a single distinct point. 
    This is because the pencils $\ell_j$ are base-point free and there is always exactly one conic on each pencil that contains a given point. 
    Similarly, it follows from \Cref{lemma:injective} that $E_{x_1}$ and $E_{x_2}$ intersect in exactly two points, namely $c$ and $c'$. 
    Now, the excess intersection formula says that
    $$D_c\cdot D_{c'}=j_\ast \left\{s(T, D_c)c(N_{D_{c'}/F})^{-1}\right\}_1,$$
    where $\{\cdot\}_1$ denotes the dimension $1$ component and $j\colon T\to F$ is the inclusion. 
    Since $T$ is reducible, we have $$\chow_1(T)=\chow_1(E_{x_1})\oplus\chow_1(E_{x_2})\oplus\bigoplus_{j\neq i,i+4}\chow_1(\ell_j).$$
    The class $s(T,D_c)$ is defined as the pushforward 
    $$s(T, D_c)=\pi'_\ast(\sum c_1(\Ox_E(k))),$$
    where $E=\proj \bigoplus_n \mathcal{I}_{T/D_c}^n/ \mathcal{I}_{T/D_c}^{n+1}$ and $\Ox_E(1)$ is a hyperplane bundle. 
    Since the irreducible components of $T$ only meet in points we can compute the excess intersection by restricting the bundles to each component. 
    Using the same arguments as in the proof of \Cref{lemma:conic-intersection}, we see that the multiplicity of the $\ell_j$ is $1$, so we only need to compute the excess intersection from $E_{x_1}$ and $E_{x_2}$.
    Denote the universal conic by $p\colon P\to F$ and the other projection by $q\colon P\to X$.

    On $E_{x_j}$, we can compute that 
    $$\{s(E_{x_j}, D_c)c(N_{D_{c'}/F})\}_1=c_1\left(\omega_F^\vee|_{E_{x_j}}\otimes \omega_{D_c}|_{E_{x_j}}\otimes \omega_{D_{c'}}|_{E_{x_j}}\otimes\omega_{E_{x_j}}^\vee\right). $$
    The first and last term are generically defined, since $E_x$ is, and pushing this forward to~$F$, we have 
    $$j_\ast\left\{s(E_{x_j},D_c)c(N_{D_{c'}/F})\right\}_1=\Gamma+  (j_c)_\ast(E_{x_j}\cdot K_{D_c})+ (j_{c'})_\ast(E_{x_j}\cdot K_{D_{c'}}),$$
    where $j_c\colon D_c\to F$ and $j_{c'}\colon D_{c'}\to F$ are the inclusions and $\Gamma$ is generically defined. 
    To compute $K_{D_c}$, consider the rational map $$D_c\dashrightarrow c,~\tilde{c}\mapsto \tilde{c}\cap c $$
    This is well-defined outside $\{c\}\cup\ell_{i+4}$ by \Cref{lemma:injective} and admits a small resolution by~$p\colon q^{-1}(c)\to D_c$ and $q\colon q^{-1}(c)\to c$. 
    We obtain a short exact sequence
    $$0\to p^\ast T_{P/X}|_{q^{-1}(c)}\to T_{q^{-1}(c)}\to q^\ast T_c\to 0.$$
    Since $c$ is a conic in $X$, we have $\Ox_X(1)|_c=\Ox_c(2)$, meaning that $q^\ast K_c=q^\ast\Ox_X(-1)|_{^{-1}(c)}$. 
    This implies that $K_{D_c}=\tilde{\Gamma}|_{D_c}-E_x$, where $\tilde{\Gamma}$ is a generically defined divisor on $F$, as~$E_x=~p_\ast q^\ast \Ox_X(1)$ and $q^{-1}(c)\to D_c$ is a small resolution.
    Furthermore, since $c$ is rational, the classes $E_x$ and $E_y$ are rationally equivalent for any two points $x,y\in c$. 
    From \Cref{lemma:injective} we know that any $c'\in E_x\cap E_y$ must lie on $\ell_{i+4}$, but there is at most one conic on a given pencil passing through any point, meaning that $E_x\cap E_y$ consists of a single point. 
    This means that $E_x^2=0$ in $D_c$. 
    Thus, $(j_c)_\ast(E_{x_j}\cdot K_{D_c})=E_x\cdot \tilde{\Gamma}$, which is itself generically defined. 
    The same argument applies to $(j_{c'})_\ast(E_{x_j}\cdot K_{D_{c'}})$.
\end{proof}
\begin{theorem}\label{them:multiplicativity+0}
    Let $Y$ be a general double EPW quartic. 
    Then the multiplication
    $$\chow^2(Y)^-_{\mathrm{hom}}\otimes\chow^2(Y)^-_{\mathrm{hom}}\to \chow_0(Y)^+_{\mathrm{hom}}$$
    is surjective.
\end{theorem}
\begin{proof}
    We know that $\chow_0(Y)^+_{\mathrm{hom}}$ is generated by cycles of the form $y+\iota(y)-2o_Y$ for~$y\in Y$. 
    Hence, we only need to show that these can be written as products of the form above. 
    Let $y\in Y$ and $c$ and $c'$ be conics such that $\alpha(c)=y$ and $\alpha(c')=\iota(y)$. 
    We assume that $c\in \ell_i$, meaning that $c'\in\ell_{i+4}$.
    Using \Cref{prop:Dc.Diota,lemma:conic-intersection} we can compute 
    \begin{align*}
        \alpha_\ast((D_c-D_{c'})^2\cdot h_F)&=\alpha_\ast(h_F\cdot (D_c^2+D_{c'}^2-2D_c\cdot D_{c'}))\\
        &=4m(y+\iota(y))-m\sum_{j\neq i, i+4} y_j -2\alpha_\ast(\Gamma\cdot h_F)\\
        &=4m(y+\iota(y))-m\sum_{j\neq i, i+4} y_j -2m o_Y,
    \end{align*}
    where $y_j=\alpha(\ell_j)$.
    The last equality follows from the fact that $\alpha_\ast(\Gamma\cdot h_F)$ is generically defined, hence must be a multiple of $o_Y$, and that the entire expression must be a homologically trivial cycle, so it suffices to count degrees to see that the coefficient must be $-2m$.
    Using \Cref{lemma:D-S} and the projection formula we can compute the left-hand side to be $(S_y-S_{\iota(y)})^2$.
    Finally, we have that
    $$(S_y-S_{\iota(y)})^2+\sum_{j=1}^4(S_{y_j}-S_{\iota(y_j)})^2=5m(y+\iota(y))-10mo_Y,$$
    completing the proof.
\end{proof}

\section{The Beauville--Voisin--Franchetta conjecture for double EPW quartics}\label{sec:BV}

The aim of this section is to prove the Beauville--Voisin--Franchetta conjecture for double EPW quartics.
The proof uses the results from \Cref{sec:incidence} and the fact that a double EPW quartic is a double cover of a quartic section in the cone $C(\IP^2\times\IP^2)$ and admits two Lagrangian fibrations, see \Cref{thm:construction} for a summary and \Cref{sec:construction-lagrangian} for more details. 

\subsection{The Beauville--Voisin--Franchetta conjecture}\label{subsec:BVF}

In the work of Beauville and Voisin \cite{BV} it was proven that the product of two divisors $D$ and $D'$ on a $K3$ surface~$S$ is a multiple of the class of any point lying on a rational curve, which is now called the Beauville--Voisin class $o_S$.
Furthermore, it was shown that $c_2(S)=24 o_S$.
This led Beauville to conjecture in \cite{beauville07} that the subring generated by divisors $\langle \chow^1(T)\rangle\subset~\chow^\ast(T)$ should map injectively into cohomology for any hyperkähler variety~$T$. 
Voisin then extended this conjecture in \cite{voisin08} to the subring generated by divisors and Chern classes.
\begin{conjecture}[Beauville--Voisin]\label{conj:BV}
    Let $T$ be a hyperkähler variety. The cycle class map restricted to the subring generated by divisors and Chern classes,
    $$\mathrm{cl}\colon\left\langle \chow^1(T), c_j(T)\right\rangle\to H^\ast(T,\IQ),$$
    is injective.
\end{conjecture}
This was proved for Hilbert schemes of $n$ points on $K3$ surfaces, where $n$ is small, and Fano varieties of lines in cubic fourfolds \cite{voisin08}, for generalised Kummer varieties \cite{fu-BV}, for double EPW sextics \cite{laterveervial, laterveer23}, and partial results are known for LLSS eightfolds \cite{bvf}.

Laterveer and Vial studied in \cite{bvf} how the Beauville--Voisin conjecture interacts with the Franchetta conjecture and showed that the generalised Franchetta conjecture and Grothendieck's Künneth standard conjecture imply the following, see \cite[Prop.~1.8]{bvf}
\begin{conjecture}[Beauville--Voisin--Franchetta conjecture]\label{conj:BVF}
    Let $\mathcal{F}$ be the moduli stack of a locally complete family of polarised hyperkähler varieties and $\mathcal{X}\to\mathcal{F}$ its universal family.
    Let $T$ be any member of the universal family and $n$ arbitrary.
    Then the cycle class map restricted to the subring of $\chow^\ast(T^n)$ generated by $\gdch^\ast_{\mathcal{F}}(T^n)$ and divisors,
    $$\mathrm{cl}\colon\left\langle \chow^1(T^n),\gdch^\ast(T^n)\right\rangle \to H^\ast(T^n,\IQ),$$
    is injective.
\end{conjecture}
 
Let $Y$ be a double EPW quartic and let $h_1$ and $h_2$ be the two divisors pulled back along the Lagrangian fibrations $Y\to\IP^2$. 
Then we know that the Franchetta property holds and that
$$\gdch^\ast(Y)=\langle h_1, h_2, c_2(Y), c_4(Y)\rangle,$$
meaning that the Beauville--Voisin conjecture and the Beauville--Voisin--Franchetta conjecture are equivalent.
In fact, the Beauville--Voisin--Franchetta conjecture for hyperkähler varieties of $K3^{[2]}$-type is equivalent to both the Beauville--Voisin conjecture and the Franchetta conjecture holding, as $\sym^2 H^2(Y,\IQ)\cong H^4(Y,\IQ)$. 
For more details on the cohomology of hyperkähler varieties, see \cite{laza-robles}.

\begin{remark}
    The Beauville conjecture, which predicts that the subring generated by divisors maps injectively into cohomology, was proved for hyperkähler varieties admitting Lagrangian fibrations in \cite{rieß}. 
    Combining this with \Cref{lemma:class-of-Z,lemma:DZ} and the Franchetta property yields the Beauville--Voisin--Franchetta conjecture for double EPW quartics.
    Our geometric approach allows for a stronger result, which we prove in the next section.
\end{remark}

\subsection{The Beauville--Voisin--Franchetta conjecture for double EPW quartics}\label{subsec:BV-quartics}

We introduce some notation.
Let $Y$ be a double EPW quartic, $\iota$ its involution, and~$Z$ the fixed locus of $\iota$.
Recall from \Cref{sec:construction-lagrangian} that $Y\to D_1^{\overline{A}}$ is a double cover over a quartic section in $C(\IP^2\times\IP^2)$.
Denote by $h_i$ the pullback of the hyperplane class of~$\IP^2$ along the two natural maps $Y\to\IP^2$. 
The invariant lattice $H^2(Y,\IZ)^\iota=\langle h_1, h_2\rangle$ is generated by~$h_1$ and $h_2$.
Denote the polarisation by $h=h_1+h_2$ and let $g=h_1-h_2$.
Furthermore, let $q$ be the Beauville--Bogomolov form.
We have $q(h)=4$ and $q(g)=-4$.
Recall that the Beauville--Voisin class $o_Y$ is represented by any point on the fixed locus~$Z$, and that $I=(\alpha\times\alpha)_\ast(I_F\cdot h_F\times h_F)$ for a choice of ample line bundle $h_F$.
We denote primitive divisors by $D$, i.e. $D\in\langle g,h\rangle^\perp$.
We begin by computing the class of $Z$ and its intersection with primitive divisors.

\begin{lemma}\label{lemma:class-of-Z}
Let $Y$ be a double EPW quartic and let $Z$ be the fixed locus of its involution. 
It is a smooth projective, regular, Lagrangian surface and its class satisfies
$$kZ=\frac{3}{4}(9h^2-g^2)-c_2(Y)$$
for some $k\in\IZ$.
\end{lemma}
\begin{proof}
Since all classes are generically defined, we may assume we are working on a general double EPW quartic and then spread out. 
In that case, the surface $Z$ is isomorphic to $D_2^{\overline{A}}$, which is the singular locus of $D_1^{\overline{A}}$ and is smooth, as $D_3^{\overline{A}}$ is empty for the general~$A$. 
Furthermore, it is Lagrangian as it is the fixed locus of an anti-symplectic involution on the hyperkähler $Y$.

Consider the blow-up $\tilde{C}\to C(\IP^2\times\IP^2)$ in the vertex of the cone. 
Projecting away from the vertex turns this into a $\IP^1$-bundle over $\IP^2\times\IP^2$, more precisely, it is isomorphic to~$\IP(\Ox(1,1)\oplus\Ox)\to\IP^2\times\IP^2$. 
Notice that the exceptional divisor is ample relative to~$\IP^2\times~\IP^2$. 
Furthermore, $D_1^{\overline{A}}$ doesn't meet the vertex and hence is contained isomorphically in $\tilde{C}$ and doesn't meet the exceptional divisor. 
Now, consider the composition~$Y\to~D_1^{\overline{A}}\to~\tilde{C}$ and denote it by $f$. 
Its derivative $Df\colon T_Y\to f^\ast T_{\tilde{C}}$ is not injective precisely on $Z$. 
Applying the Thom--Porteous formula in the form stated in \cite[Thm. 14.4]{fulton} to the vector bundles above yields a cycle supported on $Z$ with class in $\chow^2(X)$ given by~$c_2(f^\ast T_{\tilde{C}} - T_Y)$. 
As $c_1(Y)=0$, we obtain 
$$kZ=f^\ast c_2(\tilde{C})-c_2(Y).$$
Using the Euler sequence we can compute that 
$$c_2(\tilde{C})=6h_1^2+15h_1h_2+6h_2^2+6\xi(h_1+h_2),$$ 
where $\xi$ is the exceptional divisor. 
The result follows after pulling back and rearranging.

Further invariants of $Z$ are computed in \cite[Prop. 5.1 (B)]{ikkr17}. 
In the notation of \emph{loc. cit.}, the fixed locus $Z$ is $F_0$.
\end{proof}

\begin{lemma}\label{lemma:DZ}
    Let $Y$ be a double EPW quartic and let $Z$ be the fixed locus of its involution~$\iota$. 
    Then for any anti-invariant divisor $D\in\chow^1(Y)^-$, we have $D\cdot Z=0$ in~$\chow^3(Y)$.
\end{lemma}
\begin{proof}
    This is essentially \cite[Lem. 2.5]{laterveervial}.
    We have $Z\cdot\chow^1(Y)^-\subset~\chow^3(Y)^-$, but since any divisor on $Z$ is invariant, we have $Z\cdot\chow^1(Y)^-=j_\ast j^\ast\chow^1(Y)^-\subset~\chow^3(Y)^+$, where $j\colon Z\to Y$ is the inclusion.
    The result follows.
\end{proof}

Now let us recall the relations that need proving.
Given that we already know that the Franchetta conjecture holds for double EPW quartics by \Cref{thm:franchetta}, and that the subring generated by primitive divisors maps injectively into cohomology by \Cref{rmk:prim-div}, it suffices to prove the following.
For more details on the relations that hold in the cohomology of $Y$, see the discussion in \cite[Sec. 1.1]{shenvial}.
\begin{enumerate}
    \item $Dg^2=-Dh^2$ and $Dh^3=Dg^3=Dhc_2(Y)=Dgc_2(Y)=0$, $Dgh=0$, $Dc_2(Y)=\frac{30}{4}Dh^2$
    \item $D^2h^2=-D^2g^2=4q(D)o_Y$, and $D^2hg=0$, $D^2c_2(Y)=30q(D)o_Y$
    \item $D^2h = \frac{q(D)}{12}h^3$ and $D^2g = -\frac{q(D)}{12} g^3$
    \item $D^3= \frac{3q(D)}{4}Dh^2$ and $D^3h=D^3g=0$
    \item $DD'g=DD'h=0$ for any two primitive divisors $D$ and $D'$.
\end{enumerate}
Notice that many of these relations follow directly from the others by substituting and that \Cref{lemma:DZ} allows us to deduce that the relations involving $D$ and $c_2(Y)$ follow from the others.
The following few lemmas will establish the conjecture.
\begin{proposition}\label{proposition:D^2}
    Let $Y$ be a double EPW quartic and $\alpha\in\chow^2(Y)^+$ an invariant $2$-cycle.
    Then, we have that
    $$\alpha h,\alpha g\in\langle h^3, g^3\rangle\subset\chow^3(Y).$$
    In particular, it follows that $D^2h=\frac{q(D)}{12}h^3, D^2g=-\frac{q(D)}{12} g^3$ and $DD'h=DD'g=0$ for orthogonal primitive divisors $D$ and $D'$. 
    Furthermore, $D^3h=Dh^3=D^3g=Dg^3=0$.
\end{proposition}
\begin{proof}
    Let $f\colon Y\to D_1^{\overline{A}}$ be the quotient map. Then, as $\alpha h$ and $\alpha g$ are invariant, it follows that $2\alpha h=f^\ast f_\ast \alpha h$ and $2\alpha g=f^\ast f_\ast \alpha g$.
    Now, as $D_1^{\overline{A}}$ is a quartic section in~$C(\IP^2\times\IP^2)$, it follows that 
    $$f_\ast \alpha h=\tilde{h}f_\ast \alpha=\frac{1}{4}  j^\ast j_\ast f_\ast \alpha,$$
    where $j\colon D_1^{\overline{A}}\to C(\IP^2\times\IP^2)$ is the inclusion and $\tilde{h}$ is the hyperplane class on $D_1^{\overline{A}}$.
    Thus,~$\alpha h\in \langle h^3, g^3\rangle$. 
    
    A similar argument works for $\alpha g$, too. 
    Let $C(\IP(\w^2 U_1)\times~\IP(U_2))\dashrightarrow ~\IP^2$ be the rational map given by projecting to the first factor. 
    It is shown in \cite[Props.~3.8 and 3.10]{ikkr17} that the fibres of the composition $D_1^{\overline{A}}\to C(\IP^2\times\IP^2)\to\IP^2$ are Kummer quartic surfaces. 
    It is shown in \emph{loc. cit.} that for $[u_2]\in\IP(U_2)$, the Kummer surface is contained in $\IP((\w^3U_1)\oplus ((\w^2 U_1)\otimes u_2))$.
    The rational map given by the projection is resolved by $\IP(\Ox_{\IP^2}(1)^{\oplus 3}\oplus\Ox_{\IP^2})\to\IP^2$.
    In particular, $D_1^{\overline{A}}\subset \IP(\Ox_{\IP^2}(1)^{\oplus 3}\oplus\Ox_{\IP^2})$ is a fibration of Kummer quartic surfaces. 
    In this construction, $h_2$ is the pullback of the hyperplane class on $\IP(\Ox_{\IP^2}(1)^{\oplus 3}\oplus\Ox_{\IP^2})$, meaning that
    $$\alpha h_2 = \frac{1}{4}j^\ast j_\ast \alpha,$$
    where $j\colon D_1^{\overline{A}}\to\IP(\Ox_{\IP^2}(1)^{\oplus 3}\oplus\Ox_{\IP^2})$ is the inclusion.
    Then, we have that $\alpha h_2\in\langle h^3,g^3\rangle$ and hence $\alpha g\in\langle h^3,g^3\rangle$.

    Then, as the relations above hold in $H^6(Y,\IQ)$, it follows that they also hold in the Chow ring using the Franchetta property for $Y$, meaning that any relation in $H^6(Y,\IQ)$ involving $\alpha$ and $g$ and $h$ holds in $\chow^3(Y,\IQ)$.

    For the last part, we use that $c_2(Y)\alpha\beta=30q(\alpha,\beta)$ in cohomology, where $q$ is the Beauville--Bogomolov form and $\alpha,\beta\in H^2(Y,\IQ)$, see \cite{og}.
    A cohomological computation and Franchetta show that
        $$c_2(Y)h= \frac{20}{8}h^3, c_2(Y)g= \frac{20}{8} g^3.$$
    Then, we obtain the following two equations from \Cref{lemma:class-of-Z}.
    \begin{align*}
        0=DZh&=(\frac{19}{2}-\frac{20}{8})h^3D\\
        0=DZg&=-(3+\frac{20}{8})g^3D
    \end{align*}
    Finally, we obtain the last relations by substituting.
\end{proof}
\begin{lemma}\label{lemma:D^3}
    Let $Y$ be a double EPW quartic. 
    Then the relations $D^3=\frac{3q(D)}{4} Dh^2$, $Dgh=0$, and $Dg^2=-Dh^2$ hold in $\chow^3(Y)$.
    More generally,
    $$\chow^1(Y)\cdot\chow^2(Y)^+\cap\chow^3(Y)_{\mathrm{hom}}=0.$$
\end{lemma}
\begin{proof}
    We only sketch the proof of the first statement, as the others are analogous.
    We use the relation 
    $$\frac{1}{2}(\Delta_Y-\Gamma_\iota)=I\cdot\Gamma+(\iota,\id)^\ast I\cdot\Gamma'+\Gamma''$$
    obtained in \Cref{cor:proj-i}.
    As all classes that appear in the relation are generically defined, it spreads out to all double EPW quartics.
    Let $\lambda=\frac{3q(D)}{4}$.
    Since $D^3-\lambda Dh^2$ is anti-invariant, the left-hand side acts as the identity. 
    Then, we have that
    $$D^3-\lambda Dh^2=(\Gamma\cdot I)_\ast(D^3-\lambda Dh^2)+((\iota,\id)^\ast I\cdot\Gamma')_\ast(D^3-\lambda Dh^2)+\Gamma''_\ast(D^3-\lambda Dh^2).$$
    Since $\Gamma''$ is decomposable and generically defined, it acts as $0$ on homologically trivial cycles. 
    Now, we have to compute $$((\iota,\id)^\ast I\cdot\Gamma')_\ast(D^3-\lambda Dh^2)\text{ and }(\Gamma\cdot I)_\ast(D^3-\lambda Dh^2).$$ 
    By the Franchetta property, \Cref{thm:franchetta}, we know that 
    $$\Gamma,\Gamma'\in\chow^2(Y\times Y)\cap\pr_1^\ast\langle h,g,c_2(Y)\rangle\cdot\pr_2^\ast\langle h,g,c_2(Y)\rangle.$$ 
    Using the projection formula, we compute the following
    \begin{align*}
        (I\cdot\pr_1^\ast a)_\ast(D^3-\lambda Dh^2)&=I_\ast(h^2\cdot(D^3-\lambda Dh^2))=0\\
        (I\cdot\pr_1^\ast d_1\cdot\pr_2^\ast d_2)_\ast (D^3-\lambda Dh^2)&= h\cdot I_\ast(h\cdot(D^3-\lambda Dh^2))=0\\
        (I\cdot\pr_2^\ast b)_\ast(D^3-\lambda Dh^2)&=h^2\cdot I_\ast(D^3-\lambda Dh^2)=0,
    \end{align*}
    where $a,b\in\gdch^2(Y)$ and $d_1,d_2\in\gdch^1(Y)$.
    The first term vanishes by dimension reasons, the second vanishes using the relations proven in \Cref{proposition:D^2}, and the last one vanishes as $I_\ast(D^3-\lambda Dh^2)$ is a homologically trivial divisor, meaning it is zero.
    Hence $D^3-\lambda Dh^2=0$. 
    Now, let 
    $$\delta\gamma\in\chow^1(Y)\cdot\chow^2(Y)^+\cap\chow^3(Y)_{\mathrm{hom}},$$
    where $\delta$ is a divisor. 
    Then, we need to do the same computations as above, after replacing~$D^3-\lambda Dh^2$ by $\delta\gamma$. 
    The only argument that is not exactly the same is the one for
    $$a\cdot I_\ast(b \delta\gamma)=0,$$
    where $a$ is a generically defined $2$-cycle and $b$ is a generically defined divisor, i.e. a linear combination of $h$ and $g$. 
    Then, by \Cref{proposition:D^2}, we know that $b\gamma\in\langle h^3, g^3\rangle$, meaning that $b\delta\gamma$ is a homologically trivial multiple of the zero-cycle $o_Y$, i.e. $b\delta\gamma=0$, completing the proof.
\end{proof}
This is enough to prove the main result, the Beauville--Voisin--Franchetta conjecture for double EPW quartics.
\begin{theorem}\label{thm:BV}
    Let $Y$ be a double EPW quartic.
    Then the Beauville--Voisin--Franchetta conjecture holds for $Y$, meaning that the cycle class map restricted to the subring generated by divisors and generically defined cycles,
    $$\text{cl}\colon\left\langle \chow^1(Y),\gdch^\ast(Y)\right\rangle \to H^\ast(Y,\IQ),$$
    is injective.
    Furthermore, $\chow^1(Y)\cdot\chow^2(Y)^+$ and $\chow^1(Y)^{\cdot 2}\cdot\chow^2(Y)^+$ also map injectively into $H^6(Y,\IQ)$ and $H^8(Y,\IQ)$, respectively. 
\end{theorem}
\begin{proof}
    From $D^2h=\frac{q(D)}{12} h^3$ we obtain $D^2h^2=4q(D)o_Y$, as $h^4=48o_Y$. 
    The relations involving $D^2g^2$ and $D^2gh$ and $D^4$ follow in the same manner.
    It follows that any polynomial relation in $H^8(Y,\IQ)$ involving $D,h,g,$ and $c_2(Y)$ already holds in cohomology, as all the products involved are multiples of $o_Y$.
    Next, consider a polynomial relation~$[P]=~0$ in $H^6(Y,\IQ)$ involving $D,h,g,$ and $c_2(Y)$.
    This must be of degree at most $3$ in $D$, so we may decompose it as
    $$P=T+hQ+gQ'+h^2L+g^2L'+c_2(Y)L''+C$$
    where the terms are of degree $3,2,1,$ and $0$ in $D$, respectively.
    From the computations above, we obtain
    $$hQ+gQ'+C\in h\cdot\langle h^2, g^2\rangle \text{ and } T=h^2 L_1 \text{ and } h^2L+g^2L'+c_2(Y)L''=h^2(L_2+\lambda g),$$
    for some $L_1,L_2\in\langle h,g \rangle^\perp$, using $g^3=-3gh^2$. 
    Hence, $P=h^2(\mu_1 h+~\mu_2 g+~L_3)$ for some~$L_3\in\langle h, g\rangle^\perp$ and it follows that $[\mu_1 h+~\mu_2 g+~L_3]=~0$ in $H^2(Y,\IQ)$ via the Hard Lefschetz Theorem. 
    Thus, $\mu_1 h+\mu_2 g+L_3=0\in\chow^1(Y)$, as there are no homologically trivial divisors on $Y$ and $P=0$ in $\chow^3(Y)$. 
    For the last statement, we use the fact that~$\chow^1(Y)\cdot\chow^2(Y)^+=\langle h^3, g^3\rangle$, which implies that any class in~$\chow^1(Y)\cdot~\chow^1(Y)\cdot~\chow^2(Y)^+$ is a multiple of $o_Y$.
\end{proof}

It would follow from Beauville's splitting principle that the entire subring $$\left\langle \chow^1(Y), c_j(Y), \chow^2(Y)^+ \right\rangle$$ maps injectively into cohomology. 
Furthermore, it has been conjectured by Voisin in \cite{voisin16} that the subgroup of constant cycle surfaces also maps injectively into cohomology. 
We thank Robert Laterveer \cite{lat-private2} for pointing out that the subgroup of constant cycle surfaces is already contained in the subring above, reducing Voisin's conjecture to the fact that the subring also including the invariant $2$-cycles maps injectively into cohomology. 

\begin{corollary}
    Let $Y$ be a double EPW quartic and $C^2(Y)\subset \chow^2(Y)$ the subgroup generated by constant cycle surfaces. 
    Then 
    $$C^2(Y)\subset \left\langle \chow^1(Y), c_j(Y),\chow^2(Y)^+\right\rangle.$$
\end{corollary}
\begin{proof}
    Let $S$ be a constant cycle surface in $Y$. 
    We may decompose it into its invariant and anti-invariant parts $S=S^++S^-$.
    Then, since $H^2(Y,\IQ)^+\cdot H^2(Y,\IQ)^-=H^4(Y,\IQ)^-$, there exist $D_i^+\in\chow^1(Y)^+$ and $D_i^-\in\chow^1(Y)^-$ such that $S^- -\sum_i D_i^-\cdot D_i^+\in\chow^2(Y)^-_{\mathrm{hom}}$.
    Applying the relation from \Cref{cor:proj-i} and using similar arguments as in the proof of \Cref{lemma:D^3} shows that $S^- -\sum_i D_i^+\cdot D_i^-=I_\ast((S^- -\sum_i D_i^+\cdot D_i^-)\cdot a)$ for some generically defined $2$-cycle $a$. 
    Finally, $D_i^+\cdot D_i^-\cdot a$ and $S^+\cdot a$ are both multiples of $o_Y$ by \Cref{thm:BV} and $S\cdot a$ is a multiple of $o_Y$ by \Cref{prop:nef}, since $S$ is a constant cycle surface. 
    This means that $(S^- -\sum_i D_i^+\cdot D_i^-)\cdot a$ is a homologically trivial multiple of $o_Y$, hence zero, implying that $S^-=\sum_i D_i^+\cdot D_i^-$.
\end{proof}

\appendix

\section{Multiplicativity of the decomposition for double EPW sextics}\label{appendix}

In this appendix we prove analogous results to \Cref{thm:chow-iso,them:multiplicativity+0} for double EPW sextics.
The proofs are similar, but the geometric arguments involved are slightly different. 
We begin by recalling the parts of the construction of double EPW sextics that we will need. 

By $V_i$ we refer to complex vector spaces of dimension $i$.
Fix a $5$-dimensional~$V_5$. 
A Gushel--Mukai fourfold is a smooth fourfold $X$ obtained as the intersection of the Grassmannian $\Gr(2,5)$ with a quadric $Q$ and a hyperplane $H$, i.e.
$$X=\Gr(2, V_5)\cap Q\cap H\subset \IP(\w^2 V_5).$$
The hyperplane $H$ can be viewed as a two-form $\omega$ on $V_5$ with a one-dimensional kernel, which we denote by $W_1$.
It was shown in \cite{og} that one can associate to any such $X$ a dual double EPW sextic, which we denote by $\tilde{Y}^\vee$. This is a $K3^{[2]}$-type hyperk\"ahler fourfold, which admits an involution $\iota$.

Iliev and Manivel \cite{ilievmanivel} study conics in $X$, of which there are the following three types, and their Hilbert scheme $F(X)$.
Denote the plane spanned by a conic by $\langle c \rangle$.
\begin{enumerate}
    \item $\tau$-conics: the plane $\langle c\rangle$ is not contained in $\Gr(2,V_5)$. 
    \item $\rho$-conics: the plane $\langle c\rangle$ is of the form $\IP(V_3^\vee)$ and is contained in $H$.
    \item $\sigma$-conics: the plane $\langle c\rangle$ is of the form $\IP(V_1\w V_4)$ and is contained in $H$. 
\end{enumerate}
A general conic is a $\tau$-conic.
There is a morphism $\alpha\colon F(X)\to~\tilde{Y}^\vee$ that is a $\IP^1$-bundle on the complement of the loci of $\rho$- and $\sigma$-conics, which are both contracted to points~$x_\rho,x_\sigma\in\tilde{Y}^\vee$ \cite{ilievmanivel}.
For any $V_4\subset V_5$, we denote 
$$S_{V_4}=\Gr(2,V_4)\cap H\cap Q,$$
which is a degree $4$ del Pezzo surface, analogous to the $D_{(L,M)}$ in the case of double EPW quartics.
A general conic defines a unique $V_4$ and is contained in $S_{V_4}$. 
The involution~$\iota$ can be described similarly as in the case of double EPW quartics. 
Given a general~$y\in \tilde{Y}^\vee$, choosing a conic $c$ in the fibre of $\alpha$ over $y$ and choosing a hyperplane in~$S_{V_4}$ containing $c$ yields a residual conic $c'$, which satisfies $\alpha(c')=\iota(y)$. We denote the incidence variety on $F(X)$ by $I_F$ and define the class
$$I=(\alpha\times\alpha)_\ast(I_F\cdot h_F\times h_F)\text{ in }\chow^2(\tilde{Y}^\vee\times\tilde{Y}^\vee).$$

Note also that while the Franchetta property is not known to hold for double EPW sextics, there is a candidate for the Beauville--Voisin class $o_{\tilde{Y}^\vee}$, given by the class of any point lying on the fixed locus $Z\subset \tilde{Y}^\vee$ of the involution.
This is well-defined, since it was proven in \cite[Thm 1.1]{zhang24} that the fixed locus is a constant cycle subvariety of~$\tilde{Y}^\vee$.
The Beauville--Voisin conjecture was also shown to hold for double EPW sextics in \cite{laterveervial, laterveer23} with the degree zero part being exactly $\IQ o_{\tilde{Y}^\vee}$.

We may now prove the analogues of \Cref{thm:chow-iso,them:multiplicativity+0} for double EPW sextics. We will need results from \cite{zhang24} and \cite{laterveer23}.

As in the case of double EPW quartics, define the following classes. Let $y\in \tilde{Y}^\vee$ be a point and let 
$$S_y=I_\ast(y)=\frac{1}{m}\alpha_\ast(D_c\cdot h_F),$$
where $c$ is any conic with $\alpha(c)=y$ and $D_c$ is the threefold of conics intersecting $c$, $h_F$ is an ample divisor on $F(X)$ and $m$ is the intersection number of $h_F$ with the general fibre of $\alpha$. 
\begin{lemma}
    The difference 
    $$A=D_c-\alpha^\ast S_{\alpha(c)}$$
    is generically defined. 
    So is $S_y+S_{\iota(y)}$ for any $y\in \tilde{Y}^\vee$.
\end{lemma}
\begin{proof}
    The proof is analogous to that of \Cref{lemma:D-S}. 
    \Cref{pullback-I} is replaced by \cite[Lem. 4.2]{laterveer-erratum} and
    \Cref{thm:constant-cycle} by \cite[Lem.~4.3]{zhang24}. 
\end{proof}
\begin{theorem}
    Let $\tilde{Y}^\vee$ be a general double EPW sextic. 
    Then multiplication by $I_\ast(o_{\tilde{Y}^\vee})$ and $I_\ast$ induce inverse isomorphisms
    $$(\pr_1^\ast I_\ast(o_{\tilde{Y}^\vee}))_\ast\colon\chow^2(\tilde{Y}^\vee)^-_{\mathrm{hom}} \mathrel{\substack{\longrightarrow \\[-1pt] \longleftarrow}} \chow_0(\tilde{Y}^\vee)^-\colon I_\ast.$$
    In particular, homological and algebraic equivalence agree on $\chow^2(\tilde{Y}^\vee)^-$. 
\end{theorem}
\begin{proof}
    The proof is the same as the one for \Cref{thm:chow-iso} after replacing \Cref{lemma:conic-intersection} by \cite[Lem. 3.9]{zhang24}. 
\end{proof}
\begin{remark}
    We would expect multiplication by $h^2$, which is also an isomorphism by \cite[Prop. 4.1]{bolognesi-laterveer}, to agree with multiplication by $I_\ast(o_{\tilde{Y}^\vee})$ up to a scalar.
    Unfortunately, the Franchetta conjecture is not known for double EPW sextics, meaning that it is not clear that the generically defined $2$-cycle $I_\ast(o_{\tilde{Y}^\vee})$ is a linear combination of $h^2$ and $c_2(Y)$.
    Were this the case, it would follow from the computation of the class of the fixed locus in \cite[Lemma 4.1]{ferretti12} that the two isomorphisms agree.
\end{remark}

Now, we proceed as in \Cref{sec:multiplicativity} to prove the analogue of \Cref{them:multiplicativity+0}. 
We begin by computing the product $D_c\cdot D_{c'}$ for $c'$ such that $\alpha(c')=\iota(\alpha(c))$. 
Recall that a general $c$ is contained in a unique del Pezzo surface $S_{V_4}$ and that any such del Pezzo surface contains~$10$ pencils of conics, which we denote $\ell_i$. 
For any $c\in\ell_i$, we have~$\iota(\alpha(c))=~\alpha(c')$, where $c'\in\ell_{i+5}$. 
The situation is essentially the same as in the case of double EPW quartics, but for the fact that we do not restrict to $(1,1)$-conics, which explains the difference in the number of pencils. 
See \Cref{subsec:consta-cycle} for more details.
\begin{proposition}\label{prop:app-intersection}
    Let $c$ be a general conic with associated $S_{V_4}$ and assume $c\in\ell_i$.
    Let~$c'$ be a general conic lying on $\ell_{i+5}$, i.e. $\alpha(c')=\iota(\alpha(c))$.
    Then
    $$D_c\cdot D_{c'}=\Gamma+\sum_{j\neq i,i+5}\ell_j+k\ell_\sigma\text{ in }\chow^4(F),$$
    where $\Gamma$ is generically defined and $\ell_\sigma$ is a curve contained in the locus of $\sigma$-conics. 
\end{proposition}
\begin{proof}
    We compute the set-theoretic intersection $D_c\cap D_{c'}$ and then use the excess intersection formula to compute the intersection in the Chow ring. 
    First, notice that $c$ and~$c'$ intersect in two points, which we denote $x_1$ and $x_2$. 
    Let $E_x$ be the surface of conics containing a point $x\in X$,
    $$E_x=p(q^{-1}(x))=\{c\in F(X)~|~x\in c\}.$$
    Clearly, $E_{x_1}\cup E_{x_2}\subset D_c\cap D_{c'}.$ 
    We also know that any conic  $\tilde{c}\in\ell_j$ for $j\neq i,i+5$ intersects $c$ and $c'$ in one point, meaning that
    $$T\coloneq E_{x_1}\cup E_{x_2}\cup\bigcup_{j\neq i, i+5}\ell_j\subset D_c\cap D_{c'}.$$
    We claim that any conic that is not contained in $T$ is a $\sigma$-conic and that there is a~$1$-dimensional family of these. 
    The rest of the proof is analogous to the argument used in the proof of \cite[Lem. 3.9]{zhang24}. 
    Let $\tilde{c}\in D_c\cap D_{c'}$ be a conic not contained in $T$. 
    Then it cannot lie in $S_{V_4}$. 
    Let $x\in \tilde{c}\cap c$ and $y\in\tilde{c}\cap c'$. As $X\subset \Gr(2,V_5)$, these correspond to two $2$-dimensional subspaces $V_x$ and $V_y$ of $V_5$.
    Since $\tilde{c}$ is not contained in $S_{V_4}$, we have that $V_x\cup V_y$ is $3$-dimensional, meaning that $V_x\cap V_y$ is $1$-dimensional. This means that 
    $$\IP((V_x\cap V_y)\w(V_x\cup V_y))\subset \Gr(2,V_5)\cap \langle \tilde{c}\rangle.$$
    By \cite[Lem. 3.4]{zhang24}, this implies that $\tilde{c}$ is a $\rho$- or $\sigma$-conic.
    Since $c$ is general, so is~$V_4$, which means that $W_1\not\subset V_4$. 
    As $V_x\cup V_y\subset V_4$, we also have $W_1\not\subset V_x\cup V_y$, meaning that~$\langle \tilde{c}\rangle$ cannot be of the form $\IP(V_3^\vee)$ contained in $H$.
    Thus, $\tilde{c}$ is a $\sigma$-conic.

    The locus of $\sigma$-conics is isomorphic to $\mathrm{Bl}_{[W_1]}\IP(V_5)$, see \cite[Sec. 3]{ilievmanivel}.
    Let $S_c$ and~$S_{c'}$ be the surfaces swept out by lines parametrised by $c$ and $c'$ in $\IP(V_4)$. 
    Define a morphism~$S_c\cap~S_{c'}\to~\mathrm{Bl}_{[W_1]}\IP(V_5)$ by sending a point $V_1\in S_c\cap S_{c'}$ to the $\sigma$-conic obtained by intersecting the plane $\IP(V_1\w V_4)$ with $S_{V_4}$. 
    The image is precisely the set of $\sigma$-conics intersecting both $c$ and $c'$ and is $1$-dimensional. 
    Thus, conics not contained in $T$ are parametrised by a curve $\ell_\sigma$ of $\sigma$-conics.

    Now, we proceed in precisely the same way as in the proof of \Cref{prop:Dc.Diota}, as we do not need to know the multiplicity of $\ell_\sigma$ in the intersection, since it will vanish after pushing forward to $\tilde{Y}^\vee$.
    Notice that in the proof of \Cref{prop:Dc.Diota}, the generically defined class $\Gamma$ is also of the form $\Gamma'\cdot E_x$, where $\Gamma'$ is a generically defined divisor on~$F(X)$. 
\end{proof}
\begin{theorem}\label{thm:A-multiplicativity}
    Let $\tilde{Y}^\vee$ be a general double EPW sextic. 
    Then the multiplication
    $$\chow^2(\tilde{Y}^\vee)^-_{\mathrm{hom}}\otimes\chow^2(\tilde{Y}^\vee)^-_{\mathrm{hom}}\to \chow_0(\tilde{Y}^\vee)^+_{\mathrm{hom}}$$
    is surjective.
\end{theorem}
\begin{proof}
    The proof of \Cref{them:multiplicativity+0} applies almost verbatim, after replacing \Cref{prop:Dc.Diota} by \Cref{prop:app-intersection}, \Cref{lemma:conic-intersection} by \cite[Lemma 3.9]{zhang24}, and noticing the following. 
    In the proof of \Cref{them:multiplicativity+0}, we used the Franchetta property to conclude that $\alpha_\ast(\Gamma\cdot h_F)$ is a multiple of $o_Y$. 
    As we do not know whether the family of double EPW sextics satisfies the Franchetta property, we need to work around this. 
    We know that
    $$\alpha_\ast(\Gamma\cdot h_F)=\alpha_\ast(E_x\cdot \Gamma'\cdot h_F),$$ where $\Gamma'$ is a generically defined divisor. 
    Since $\alpha$ is a $\IP^1$-bundle outside the locus of $\rho$- and $\sigma$-conics, we know that 
    $$\Gamma'=\alpha^\ast d_1+a_1\xi+b_2[F^\sigma]\text{ and }h_F=\alpha^\ast d_2+a_2\xi+b_2[F^\sigma],$$
    where $d_1,d_2\in\chow^1(\tilde{Y}^\vee)$, $\xi$ is the class of the relative hyperplane bundle of the $\IP^1$-bundle and $[F^\sigma]$ is the class of the locus of $\sigma$-conics.
    Since $F^\sigma$ is contracted to a point by $\alpha$, any multiple of $[F^\sigma]$ vanishes after pushing forward.
    Furthermore, the conics contained in the fibre of $\alpha$ over a general point $y\in \tilde{Y}^\vee$ are disjoint, meaning that $\xi\cdot E_x=0$.
    Thus, we have $$\alpha_\ast(\Gamma\cdot h_F)=\alpha_\ast(E_x)\cdot d_1\cdot d_2.$$
    Now, given $y\in\alpha(E_x)$, there exists a conic $c\in\ell_i$ with $\alpha(c)=y$ and $x\in c$.
    We know that there exists $c'\in\ell_{i+5}$ with $x\in c'$, which means that $\alpha(c')=\iota(y)$ and $\iota(y)\in\alpha(E_x)$. 
    Thus, $\alpha_\ast(E_x)$ is $\iota$-invariant and $\alpha_\ast(E_x)\cdot d_1\cdot d_2$ is a multiple of $o_{\tilde{Y}^\vee}$ using the main theorem of \cite{laterveer23}. With this, the rest of the proof goes through as in the case of double EPW quartics. 
\end{proof}

\addtocontents{toc}{\protect\setcounter{tocdepth}{-1}}
\section*{Glossary of Notation}
\addtocontents{toc}{\protect\setcounter{tocdepth}{1}}

\begin{tabbing}
  \hspace{3.2cm} \= \hspace{12cm} \= \kill
  $Y$ \> a double EPW quartic \>  \ref{sec:construction} \\
  $\mathcal{U}$ \> the $19$-dimensional family of double EPW quartics \>  \ref{sec:construction} \\
  $q$ \> the Beauville--Bogomolov form \>  \ref{sec:construction} \\
  $\iota$ \> the involution on $Y$ \>  \ref{sec:construction} \\
  $U_1, U_2$ \> $3$-dimensional complex vector spaces     \>  \ref{sec:construction-verra}  \\
  $C(T)$  \> the projective cone over a variety $T$     \>  \ref{sec:construction-verra}  \\
  $X$  \> a Verra fourfold     \>   \ref{sec:construction-verra} \\
  $Q$\> the quadric defining a Verra fourfold     \>  \ref{sec:construction-verra}  \\
  $\pi\colon X\to\IP^2\times\IP^2$  \> the double cover associated to $X$     \>  \ref{sec:construction-verra}  \\
  $c$  \> a $(1,1)$-conic in $X$     \>  \ref{sec:construction-verra} \\
  $F(X)$  \> the Hilbert scheme of $(1,1)$-conics in $X$     \> \ref{sec:construction-verra}   \\
  $I_{X,2}$  \> $H^0(X,\mathcal{I}_{X}(2)),$ quadrics containing $X$     \> \ref{sec:construction-verra}  \\
  $\mathfrak{Q}$  \> a quadric containing $X$     \> \ref{sec:construction-verra}  \\
  $\langle c\rangle$  \> the plane spanned by a conic $c$     \> \ref{sec:construction-verra}  \\
  $L,M$ \> lines in $\IP^2$     \>  \ref{sec:construction-verra}  \\
  $D_{(L,M)}$  \> $\pi^{-1}(L\times M)$, a del Pezzo surface of degree $4$     \>   \ref{sec:construction-verra} \\
  $\tilde{Q}$  \> the unique quadric threefold containing $D_{(L,M)}\cup\langle c\rangle$     \>  \ref{sec:construction-verra}  \\
  $\alpha\colon F(X)\to Y$  \> the $\IP^1$-bundle over $Y$ associated to $X$    \>  \ref{sec:construction-verra}  \\
  $S_i, S$  \> $K3$ surfaces     \>  \ref{sec:construction-moduli}  \\
  $\beta_i, \beta$  \> Brauer classes on $K3$ surfaces     \>   \ref{sec:construction-moduli}  \\
  $\mathcal{M}_{(0,H,0)}(S,\beta)$  \> a moduli space of stable objects on a twisted $K3$ surface     \> \ref{sec:construction-moduli}  \\
  $\mathrm{LG}(10,\w^3 V)$  \> the Grassmannian of Lagrangian subspaces     \> \ref{sec:construction-lagrangian}  \\
  $D_k^{\overline{A}}$  \> the Lagrangian degeneracy loci associated to $A$     \>  \ref{sec:construction-lagrangian}  \\
  $\mathcal{X}$\> a family of projective varieties     \> \ref{sec:franchetta}  \\
  $\gdch^\ast_B(\mathcal{X}_0)$  \> generically defined cycles with respect to the family $\mathcal{X}\to B$     \> \ref{sec:franchetta}  \\
  $\mathcal{F}$  \> a moduli stack of a locally complete family of polarised     \> \ref{sec:franchetta}  \\
  \> hyperkähler varieties \> \\
  $\mathcal{K}_2$  \> the moduli of degree $2$ $K3$ surfaces     \> \ref{sec:franchetta}  \\
  $o_Y, o_S$  \> the Beauville--Voisin class      \> \ref{sec:franchetta}  \\
  $h_1, h_2$  \> the divisors associated to the Lagrangian fibrations on $Y$     \> \ref{sec:franchetta}  \\
  $P$  \> the universal conic $P\subset F(X)\times X $    \> \ref{sec:incidence1}  \\
  $I_F$  \> the incidence variety of conics in $F(X)\times F(X)$     \> \ref{sec:incidence}  \\
  $q\colon P\to X$  \> the projection to $X$     \> \ref{sec:incidence1}  \\
  $p\colon P\to F(X)$ \> the projection to $F(X)$ \>  \ref{sec:incidence1}\\
  $\Gamma$, $\Gamma'$, $\Gamma''$  \> decomposable and generically defined cycles     \> \ref{sec:incidence1}  \\
  $h_F$\>  an ample divisor on $F(X)$    \> \ref{sec:incidence1}  \\
  $m$\>  the intersection number of $h_F$ with the general fibre of $\alpha$    \> \ref{sec:incidence1}  \\
  $I$  \> the push-forward of $I_F$ to $Y\times Y$     \> \ref{sec:incidence2}  \\
  $Z$  \>  the fixed locus of the involution $\iota$    \> \ref{sec:incidence2}  \\
  $D_c$  \> the threefold of conics intersecting a given conic $c$     \> \ref{subsec:consta-cycle}  \\
  $\ell_i$  \> a pencil of conics in a $D_{(L,M)} $    \> \ref{subsec:consta-cycle}  \\
  $\Psi$  \> the correspondence induced by the universal conic     \> \ref{subsec:consta-cycle}  \\
  $S_y$  \> the class $I_\ast(y)$     \> \ref{subsec:constant2}  \\
  $\chow^\ast(Y)^{+/-}$\>  the group (anti-)invariant cycles \> \ref{subsec:constant2} \\
  $\chow^\ast(T)_{\mathrm{hom}}$\> homologically trivial cycles     \> \ref{subsec:constant2}  \\
  $g$  \> the divisor $g=h_1-h_2$     \> \ref{subsec:motives}  \\
  $F^\bullet$  \> the conjectural Bloch--Beilinson filtration     \> \ref{sec:multiplicativity}  \\
  $E_x$  \> the class of conics passing through a point $x\in X$     \> \ref{sec:multiplicativity}  \\
  $h$  \> the polarisation $h=h_1+h_2$ of $Y$     \> \ref{sec:BV}  \\
  $D$  \> a primitive divisor in $Y$     \> \ref{sec:BV}  \\
  $\tilde{Y}^\vee$ \> a dual double EPW sextic \> \ref{appendix} 

\end{tabbing}

\printbibliography[title=References]

\end{document}